 \documentclass[pdflatex,sn-apa, iicol]{sn-jnl}

\usepackage{graphicx}%
\usepackage{multirow}%
\usepackage{amsmath,amssymb,amsfonts}%
\usepackage{amsthm}%
\usepackage{mathrsfs}%
\usepackage[title]{appendix}%
\usepackage{xcolor}%
\usepackage{textcomp}%
\usepackage{manyfoot}%
\usepackage{booktabs}%
\usepackage{algorithm}%
\usepackage{algorithmicx}%
\usepackage{algpseudocode}%
\usepackage{listings}%
\usepackage{bm}%
\usepackage{nicefrac}

\usepackage{xurl}
\usepackage{tikz, pgfplots, subfig, xcolor, mathtools}
\usetikzlibrary{patterns,backgrounds,fit}

\definecolor{cj1}{rgb}{0.9, 0.9, 0.9}
\definecolor{cj2}{rgb}{0.85, 0.85, 0.85}
\definecolor{cj3}{rgb}{0.7, 0.7, 0.7}
\definecolor{cj4}{rgb}{0.55, 0.55, 0.55}

\newcommand{\Lps}{\mathcal{L}} 
\newcommand{\Sched}{\sigma}	
\newcommand{\Pg}{\mathcal{P}}	
\newcommand{\Wfl}{\mathcal{W}} 
\newcommand{\Sx}{\mathcal{J}}	
\newcommand{\St}{\mathcal{S}}	
\newcommand{\FDP}{V}	
\newcommand{\Rto}{\rho}	 
\newcommand{\OFV}{Z}	 
\newcommand{\OFVV}{\phi}	 
\newcommand{\Jbs}{\mathcal{J}} 

\theoremstyle{thmstyleone}%
\newtheorem{theorem}{Theorem}[section]
\newtheorem{example}[theorem]{Example}%
\newtheorem{definition}[theorem]{Definition}%
\newtheorem{lemma}[theorem]{Lemma}%
\newtheorem{corollary}[theorem]{Corollary}%

\tikzset{job/.style={rectangle,draw,anchor=west,minimum height=0.75cm,minimum width=0.75cm}}				
\tikzset{job1/.style={rectangle,draw,anchor=west,minimum height=0.75cm,minimum width=0.75cm}}
\tikzset{job2/.style={rectangle,draw,anchor=west,minimum height=0.75cm,minimum width=0.75cm, pattern = north east lines}}				
\tikzset{job3/.style={rectangle,draw,anchor=west,minimum height=0.75cm,minimum width=0.75cm, pattern = vertical lines}}
\tikzset{job4/.style={rectangle,draw,anchor=west,minimum height=0.75cm,minimum width=0.75cm, pattern = dots}}
\tikzset{job5/.style={rectangle,draw,anchor=west,minimum height=0.75cm,minimum width=0.75cm, pattern = horizontal lines}}
\tikzset{job_gr/.style={rectangle,draw,anchor=west,minimum height=0.75cm,minimum width=0.75cm, opacity=0.05}}
\tikzset{job_tr/.style={rectangle,anchor=west,minimum height=0.75cm,minimum width=0.75cm}}

\tikzset{lab/.style={fill=white,inner sep=1pt}}
\tikzset{lab2/.style={fill=white,inner sep=1pt, opacity=0.1}}

\begin{document}

\title{Flow shops with reentry:}
\subtitle{The total weighted completion time objective}


\author[1]{\fnm{Maximilian} \sur{von Aspern}}\email{maximilian.aspern@tum.de}

\author[1]{\fnm{Felix} \sur{Buld}}\email{felix.buld@tum.de}

\author*[2]{\fnm{Nicklas} \sur{Klein}}\email{nicklas.klein@unibe.ch}

\author[3]{\fnm{Michael} \sur{Pinedo}}\email{mlp5@stern.nyu.edu}

\affil[1]{\orgdiv{Chair of Operations Research}, \orgname{Technical University of Munich}, \orgaddress{\street{Arcisstr. 21}, \city{80333~Munich}, \country{Germany}}}

\affil*[2]{\orgdiv{Department of Business Administration}, \orgname{University of Bern}, \orgaddress{\street{Engehaldenstrasse~4}, \city{3012~Bern}, \country{Switzerland}}}

\affil[3]{\orgdiv{Leonard N. Stern School of Business}, \orgname{New York University}, \orgaddress{\street{44 W 4th St.}, \city{New~York}, \state{NY}, \postcode{10012}, \country{USA}}}

\abstract{
	Flow shops are widely studied machine environments in which all jobs must visit all machines in the same order.	While conventional flow shops assume that each job traverses the shop only once, many industrial environments require jobs to loop through the shop multiple times before completion. This means that after traversing the shop and completing its processing on the last machine, a job must return to the first machine and traverse the shop again until it has completed all its required loops. Such a setting, referred to as a flow shop with reentry, has numerous applications in industry, e.g., semiconductor manufacturing. The planning problem is to schedule all loops of all jobs while minimizing the total weighted completion time. In this paper, we consider reentrant flow shops with unit processing times. We show that this problem is strongly NP-hard if the number of machines is part of the input. We propose the Least Remaining Loops First (LRL) priority rule and show that it minimizes the total unweighted completion time. Then, we analyze the Weighted Least Remaining Loops First (WLRL) priority rule and show that it has a worst-case performance ratio of $(1+\sqrt{2})/2$ (about 1.2). Additionally, we present a fully polynomial time approximation scheme (FPTAS) and a pseudo-polynomial time algorithm if the number of machines in the flow shop is fixed. 
}


\maketitle

\section{Introduction}\label{sec:Intro}


Flow shops are widely studied machine environments in which all jobs must visit all machines in the same order. 
In a conventional flow shop with~$m$ machines in series and~$n$ jobs, it is assumed that each job traverses the shop only once. 
However, many industrial environments require the jobs to undergo the same set of sequential processing steps multiple times. 
Such an environment is called a flow shop with reentry and can be described as follows. 
Each job must be processed first on machine~$1$, then on machine~$2$, and so on until it is processed on machine~$m$. 
After completing its processing on machine~$m$, the job must return to machine~$1$ and traverse the shop again until it has completed all its required loops. 
The planning problem is to schedule all loops of all jobs while adhering to precedence relations between consecutive loops of the same job. 
This means that a loop of a job must complete its processing on machine~$m$ before the next loop of that job can start its processing on machine~$1$. 


In this paper, we consider the problems of minimizing the total completion time and total weighted completion time for reentrant flow shops. 
Minimizing these objective functions is of interest, e.g., if there is limited storage capacity or if jobs must undergo subsequent processing steps.  
It is well known that minimizing the total completion time for conventional flow shops with~$m\geq 2$ machines is NP-hard. 
We show that for reentrant flow shops with~$m\geq 2$ machines, minimizing the total weighted completion time is strongly NP-hard even if the jobs have unit processing times.  
Then, we analyze two priority rules. 
First, we analyze the ``Least Remaining Loops First'' (LRL) priority rule. 
This priority rule schedules, whenever machine~$1$ becomes available, a loop of a job with the least remaining loops among all available jobs. 
We show that this priority rule minimizes the total unweighted completion time as well as the total weighted completion time if the weights and numbers of loops are ``agreeable'', i.e., jobs with high weights also have small numbers of loops. 
Second, we analyze the ``Weighted Least Remaining Loops First'' (WLRL) priority rule. 
This priority rule schedules a loop of a job with the highest ratio of weight to the number of remaining loops among all available jobs whenever machine~$1$ becomes available.  
While it is not guaranteed that this priority rule generates optimal schedules, we show that it has a theoretical worst-case performance ratio of~$\frac{1}{2}(1+\sqrt{2})$ $(\approx 1.2)$.
In practice, the WLRL rule seems to perform even better. 
In a numerical study, we observe that the average relative deviation of the total weighted completion time under WLRL from that under the optimal schedule is about~$1\%$. 
Finally, we analyze reentrant flow shops with a fixed number of machines. 
We show that this problem is weakly NP-hard and propose a dynamic programming algorithm that solves the problem in pseudo-polynomial time.
Additionally, we present a fully polynomial-time approximation scheme (FPTAS), which produces arbitrarily good approximate solutions. 


Reentrant flow shops can be found in many manufacturing facilities across various industries. 
For example, the repeated dyeing and drying process in textile manufacturing and the coating and drying process in mirror manufacturing can be modeled as reentrant flow shops \citep{ChoiKim2007}. 
Also, printed circuit board manufacturing lines require job recirculation \citep{Pan2003}. 
One of the most frequently discussed applications of job recirculation is in semiconductor wafer fabs \citep{pfund2006,Monch2011surveySemiconductors}. 
In these manufacturing facilities, wafers undergo a series of sequential processing steps that are repeated multiple times. 
\cite{Danping2011} survey research methodologies for reentrant shops with a particular focus on semiconductor wafer manufacturing. 
\cite{Wu2022} model the so-called furnace (a part of the semiconductor manufacturing process) as a reentrant shop with additional constraints; 
they develop a scheduling heuristic and describe an industry implementation of their algorithm in a real-world wafer fab. 
In addition to manufacturing environments that involve planned repetitive processing, job recirculation may also occur in any production system if a job has to undergo repair or rework after a defect has been detected \citep{Danping2011}.


Flow shops with reentry were first examined by \cite{Graves1983}. 
Since then, they have received considerable attention in the literature. 
Most publications focus on minimizing the makespan \citep{LevAdiri1984_V_shop, Wang1997, Pan2003, Chen2006, ChoiKim2007, Choi2008, Shufan2023, Zheng2023}. 
These papers discuss various formats of job reentry. For example, V-shops where the machine route is $1 \rightarrow 2 \rightarrow \ldots \rightarrow m-1 \rightarrow m \rightarrow m-1 \rightarrow \ldots \rightarrow 1$ \citep{LevAdiri1984_V_shop} or chain-reentrant shops in which only the first machine is revisited \citep{Wang1997}. 
For an overview of the various different formats of job reentry, see, e.g., \cite{Emmons2012}. 
The reentrant flow shop with full loops is the most frequently examined machine route. 
In such an environment, each job must visit the machines in each loop in the order $1 \rightarrow 2 \rightarrow \ldots \rightarrow m$ \citep{Pan2003,Chen2006,ChoiKim2007,Choi2008}. 
\cite{Yu2020} investigate machine-ordered and proportionate flow shops with full loop reentry. 
In these settings, the processing times of jobs depend only on either the machine (machine-ordered) or the job (proportionate).  
The authors identify a class of schedules that minimize the makespan if all jobs have the same number of loops. 
Additionally, they propose the ``Most Remaining Loops First'' (MRL) priority rule and show that it minimizes the makespan under certain conditions, which include the case with unit processing times. 
The literature concerned with the total completion time objective is more scarce. 
\cite{Kubiak1996} examine the total completion time objective for a reentrant flow shop with a ``hub'' machine. 
In this setting, the jobs must alternate between the hub machine and the $m-1$ other machines, i.e., $1\rightarrow 2 \rightarrow 1 \rightarrow 3 \rightarrow \ldots \rightarrow m \rightarrow 1$. 
\cite{Jing2011} propose a k-insertion heuristic for minimizing the total completion time in a reentrant flow shop with full loops. 
However, to the best of our knowledge, there is no analysis of optimal scheduling policies for the total completion time objective in the literature.


The remainder of this paper is organized as follows. 
Section~\ref{sec:Problem} describes the planning problem in detail and provides an example. 
In Section~\ref{sec:Preliminiaries}, we introduce two concepts which we call progressions and non-interruptive schedules. We show that there is at least one optimal schedule that is non-interruptive. 
In Section~\ref{sec:m_input}, we show that minimizing the total weighted completion time is strongly NP-hard if the number of machines is part of the input. We analyze the LRL rule and show that it minimizes the total completion time.
Then, we present the WLRL rule and show its performance guarantee.
In Section~\ref{sec:m_fixed}, we show that the problem remains NP-hard for a fixed number of machines, and present the pseudo-polynomial time algorithm and the FPTAS.
In Section~\ref{sec:Conclusion}, we conclude the paper and provide an outlook on future research.

\section{Problem description and notation}\label{sec:Problem}

Consider a flow shop with~$m$ machines in series and $n$ jobs. 
Job $j$ traverses the shop~$\Lps_j$ times. 
Let $p_{ijk}$ be the processing time of job~$j$ on machine~$i$ in loop~$k$, $k= 1,\ldots,\Lps_j$ and let $w_j>0$ denote the weight of job~$j$.
Let $\ell_{jk}$ refer to loop~$k$ of job~$j$. 
For a schedule $\Sched$, let $S_{jk}$ denote the start time of job~$j$ in loop~$k$ on machine~$1$ and~$C_{jk}$ the completion time on machine~$m$. If the schedule is not clear from the context, we may instead write $S_{jk}(\Sched)$ and $C_{jk}(\Sched)$, respectively.
Job~$j$ is completed once its last loop is completed, i.e., $C_j = C_{j\Lps_j}$. 
In what follows, we focus our analysis on permutation schedules, which are defined by considering each loop~$k$ of job~$j$ as a sub-job of job~$j$ \citep{Yu2020}. 
In total, $\sum_{j=1}^n \Lps_j$ loops must be sequenced while adhering to precedence constraints between successive loops of the same job. 
These precedence constraints take the form $S_{jk} \geq C_{j,k-1}$ for all jobs~$j=1,\ldots,n$ and all loops~${k=2,\ldots,\Lps_j}$. 
This means the next loop of a job can start on machine~$1$ either immediately upon completion of the job's previous loop on machine~$m$ or at some later point in time. 
Our objective is to determine a schedule that minimizes the total weighted completion time $\sum_{j=1}^n w_jC_j$. 

We assume throughout this paper that~$p_{ijk}=p=1$, motivated, e.g., by some underlying takt time in manufacturing. 
This implies that after completing its processing on machine~$1$, a job immediately continues its processing on machine~$2$, then on machine~$3$, and so on, until it completes its processing on machine~$m$. 
It also implies that an optimal schedule does not allow for unforced idleness, i.e., in an optimal schedule, machine~$1$ is only left idle if no loop of any job is available to begin its processing on machine~$1$. 
%
%
%
%
We denote this scheduling problem as $F|reentry, \ p_{ijk} = 1|\sum w_jC_j$ using the standard three-field scheduling notation \citep{Graham1979_ThreeField}. 

We illustrate the scheduling problem through the following example.

\begin{example}\label{ex:IllExample}
	Consider a reentrant flow shop with $m=3$ machines and $n=5$ jobs. The jobs require $\Lps = (2,2,2,3,4)$ loops. The jobs have weights~$w = (2,1,1,3,4)$. 
	The Gantt chart for permutation sequence $\Sched = [\ell_{51};\allowbreak\ell_{41};\allowbreak\ell_{11};\allowbreak\ell_{21};\allowbreak\ell_{31};\allowbreak\ell_{42};\allowbreak\ell_{22};\allowbreak\ell_{32};\allowbreak\ell_{52};\allowbreak\ell_{12};\allowbreak\ell_{43};\allowbreak\ell_{53};\allowbreak\ell_{54}]$ is displayed in Figure~\ref{fig:ExampleGanttChart}. 
	We observe that machine~$1$ is idle for two time units as loop $\ell_{54}$ can only start on machine~$1$ after loop $\ell_{53}$ is completed on machine~$3$. 
	The total weighted completion time of this schedule is $\sum_{j=1}^5 {w_jC_j = 150}$. 
\end{example}

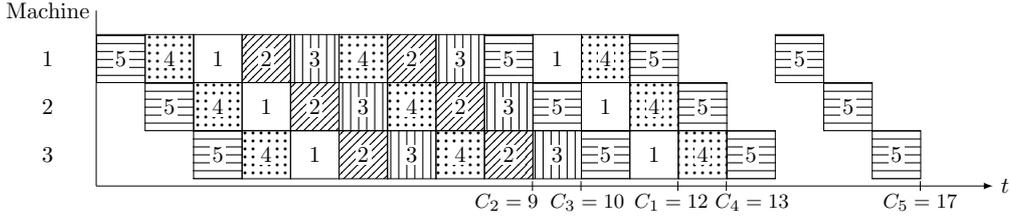
\begin{figure*}
	\centering
	\caption{Gantt chart for the schedule of Example~\ref{ex:IllExample}}\label{fig:ExampleGanttChart}
	\scalebox{0.85}{
	\begin{tikzpicture}[scale=0.75]
			
			\pgfmathsetmacro{\m}{3.65}
			
			\draw[-] (0,0) -- (0, -\m);
			\draw[-latex] (0, -\m) -- node[pos=1,right] {$t$} (18.5, -\m);
			
			\node (Machines) at (-1,0) {Machine};
			\node (Mone) at (-1,-1) {$1$};
			\node (Mtwo) at (-1,-2) {$2$};
			\node (Mthree) at (-1,-3) {$3$};
			
			\foreach \x in {2, 9}{
				\node[job1] (1) at (\x,-1) {}; \node at (1.center) [lab] {$1$};
				\node[job1] (1) at (\x+1,-2) {}; \node at (1.center) [lab] {$1$};
				\node[job1] (1) at (\x+2,-3) {}; \node at (1.center) [lab] {$1$};
			}
			\foreach \x in {3,6}{
				\node[job2] (2) at (\x,-1) {}; \node at (2.center) [lab] {$2$};
				\node[job2] (2) at (\x+1,-2) {}; \node at (2.center) [lab] {$2$};
				\node[job2] (2) at (\x+2,-3) {}; \node at (2.center) [lab] {$2$};
			}
			\foreach \x in {4,7}{
				\node[job3] (3) at (\x,-1) {}; \node at (3.center) [lab] {$3$};
				\node[job3] (3) at (\x+1,-2) {}; \node at (3.center) [lab] {$3$};
				\node[job3] (3) at (\x+2,-3) {}; \node at (3.center) [lab] {$3$};
			}
			\foreach \x in {1,5,10}{
				\node[job4] (4) at (\x,-1) {}; \node at (4.center) [lab] {$4$};
				\node[job4] (4) at (\x+1,-2) {}; \node at (4.center) [lab] {$4$};
				\node[job4] (4) at (\x+2,-3) {}; \node at (4.center) [lab] {$4$};
			}		
			\foreach \x in {0,8,11,14}{
				\node[job5] (5) at (\x,-1) {}; \node at (5.center) [lab] {$5$};
				\node[job5] (5) at (\x+1,-2) {}; \node at (5.center) [lab] {$5$};
				\node[job5] (5) at (\x+2,-3) {}; \node at (5.center) [lab] {$5$};
			}

			\foreach \x in {12}{
				\node[anchor=north,xshift=-1mm] (C4) at (\x,-\m) {\small $C_1=\x$};
				\draw (\x,-0.1-\m) -- (\x,0.1-\m);
			}
			\foreach \x in {9}{
				\node[anchor=north,xshift=-4mm] (C4) at (\x,-\m) {\small $C_2=\x$};
				\draw (\x,-0.1-\m) -- (\x,0.1-\m);
			}
			\foreach \x in {10}{
				\node[anchor=north,xshift=1mm] (C4) at (\x,-\m) {\small $C_3=\x$};
				\draw (\x,-0.1-\m) -- (\x,0.1-\m);
			}
			\foreach \x in {13}{
				\node[anchor=north,xshift=4mm] (C4) at (\x,-\m) {\small $C_4=\x$};
				\draw (\x,-0.1-\m) -- (\x,0.1-\m);
			}
			\foreach \x in {17}{
				\node[anchor=north] (C4) at (\x,-\m) {\small $C_5=\x$};
				\draw (\x,-0.1-\m) -- (\x,0.1-\m);
			}
			
	\end{tikzpicture}}
\end{figure*}

\section{Preliminaries and non-interruptive schedules}\label{sec:Preliminiaries}

In this section, we introduce some preliminary definitions and lemmas. 
First, we introduce two concepts we refer to as progressions and progression interchanges. 
Utilizing these concepts, we introduce so-called non-interruptive schedules and show that there is a non-interruptive schedule that minimizes the total weighted completion time.

A progression can be described as an ordered set of loops of one or more jobs that are processed consecutively. 
This means that the next loop in the progression starts on machine~$1$ as soon as the previous loop completes on machine~$m$. 
For a given schedule, a progression is initiated and uniquely defined by its first loop. 
A progression ends as soon as there is an idle time on machine~$1$ at the completion of the last loop in the progression. 



\begin{definition}
	\label{def: progression}
	A progression initiated by loop $\ell_{jk}$ is defined as the following ordered set of loops: 
	\begin{align*}
		\Pg\left(\ell_{jk}\right) \coloneqq \left\{\ell_{jk}\right\} & \cup \left\{\ell_{j'k'}: S_{j'k'} = C_{jk}\right\} \\
        & \cup \left\{\ell_{j''k''}: S_{j''k''} = C_{j'k'}\right\} \cup \ldots
	\end{align*}
	
	Note that we could also characterize a progression via loops that have a difference in start times of $m$: ${\Pg\left(\ell_{jk}\right) \coloneqq \left\{ \ell_{j'k'}:  S_{j'k'} = S_{jk} + rm, r \in \mathbb{N}_0\right\}}$. 
    If the loop that initiates the progression is not relevant or clear from the context, we simply denote the progression as $\Pg$. 
	Consider the set of jobs that have their last loop in progression~$\Pg$, denoted as $\{j: \ell_{j\Lps_j}\in \Pg \}$. We denote the total weight of these jobs as $\Wfl(\Pg) \coloneqq \sum_{j: \ell_{j\Lps_j}\in \Pg} w_j$.  
\end{definition}

	Note that the loops within one progression may belong to different jobs. 
	Also, loops of the same job may be in different progressions. 
	The following example demonstrates what the pattern of a progression may look like. 

\begin{example}\label{ex:ExampleProgression_1}
	Consider the same schedule as in Example~\ref{ex:IllExample}. The progression initiated by loop~$\ell_{41}$ is $\Pg(\ell_{41})=\{\ell_{41},\ell_{31},\ell_{32},\ell_{43}\}$ with ${\Wfl\big(\Pg(\ell_{41})\big)=4}$. We display the progression in Figure~\ref{fig:ExampleProgression_1}. 
\end{example}

\begin{figure*}
	\centering
	\caption{Gantt chart for progression $\Pg(\ell_{41})$}\label{fig:ExampleProgression_1}
	\scalebox{0.85}{
		\begin{tikzpicture}[scale=0.75]

			\pgfmathsetmacro{\m}{3.65}
			
			\draw[-] (0,0) -- (0, -\m);
			\draw[-latex] (0, -\m) -- node[pos=1,right] {$t$} (17.5, -\m);
			

			\node (Machines) at (-1,0) {\small Machine};
			\foreach \x in {1,2,3}{
				\node (Mone) at (-1,-\x) {$\x$};
			}
		
			\foreach \x in {2, 9}{
				\node[job_gr] (1) at (\x,-1) {}; \node at (1.center) [lab2] {$1$};
				\node[job_gr] (1) at (\x+1,-2) {}; \node at (1.center) [lab2] {$1$};
				\node[job_gr] (1) at (\x+2,-3) {}; \node at (1.center) [lab2] {$1$};
			}
			\foreach \x in {3,6}{
				\node[job_gr] (2) at (\x,-1) {}; \node at (2.center) [lab2] {$2$};
				\node[job_gr] (2) at (\x+1,-2) {}; \node at (2.center) [lab2] {$2$};
				\node[job_gr] (2) at (\x+2,-3) {}; \node at (2.center) [lab2] {$2$};
			}
			\foreach \x in {5}{
				\node[job_gr] (4) at (\x,-1) {}; \node at (4.center) [lab2] {$4$};
				\node[job_gr] (4) at (\x+1,-2) {}; \node at (4.center) [lab2] {$4$};
				\node[job_gr] (4) at (\x+2,-3) {}; \node at (4.center) [lab2] {$4$};
			}		
			\foreach \x in {0,8,11,14}{
				\node[job_gr] (5) at (\x,-1) {}; \node at (5.center) [lab2] {$5$};
				\node[job_gr] (5) at (\x+1,-2) {}; \node at (5.center) [lab2] {$5$};
				\node[job_gr] (5) at (\x+2,-3) {}; \node at (5.center) [lab2] {$5$};
			}		
			
			\foreach \x in {4,7}{
				\node[job3] (3) at (\x,-1) {}; \node at (3.center) [lab] {$3$};
				\node[job3] (3) at (\x+1,-2) {}; \node at (3.center) [lab] {$3$};
				\node[job3] (3) at (\x+2,-3) {}; \node at (3.center) [lab] {$3$};
			}
			\foreach \x in {1,10}{
				\node[job4] (4) at (\x,-1) {}; \node at (4.center) [lab] {$4$};
				\node[job4] (4) at (\x+1,-2) {}; \node at (4.center) [lab] {$4$};
				\node[job4] (4) at (\x+2,-3) {}; \node at (4.center) [lab] {$4$};
			}		
	\end{tikzpicture}}
\end{figure*}
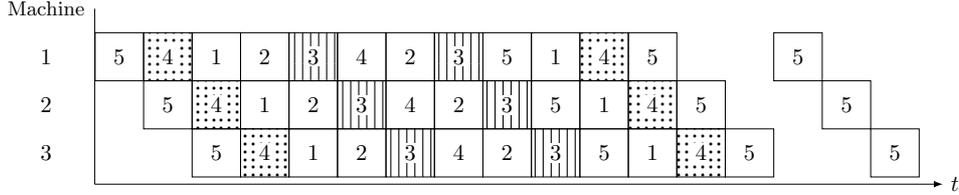


We now introduce progression interchanges. 
Informally, a progression interchange generates a new schedule in which all loops of one progression are interchanged with all loops of another progression. All loops that are not part of the interchange remain untouched. 

\begin{definition}
	Consider two progressions $\Pg(\ell_{j_1k_1})$ and $\Pg(\ell_{j_2k_2})$ in the schedule $\Sched$ with ${S_{j_1k_1}<S_{j_2k_2}}$ and ${\ell_{j_2k_2}\notin \Pg(\ell_{j_1k_1})}$. 
	Denote the difference in start times as $\delta=S_{j_2k_2}-S_{j_1k_1}$.
	The new schedule $\Sched'$ obtained through the progression interchange is defined by the following start times of loops:
	\begin{align*}
	S_{jk} \big ( \Sched' \big ) = \begin{cases}
		S_{jk} + \delta, &\text{ if } \ell_{jk}\in \Pg(\ell_{j_1k_1})\\
		S_{jk} - \delta, &\text{ if } \ell_{jk}\in \Pg(\ell_{j_2k_2})\\
		S_{jk}, &\text{ otherwise.}
	\end{cases}
	\end{align*}
	
	The total weighted completion time of schedule~$\Sched'$ is: 
	\begin{align*}
	\sum_{j=1}^n w_jC_j\big(\Sched' \big ) = & \sum_{j=1}^n w_jC_j\big(\Sched \big )\\
    & + \delta \big [  \Wfl\big(\Pg(\ell_{j_1k_1})\big) - \Wfl\big(\Pg(\ell_{j_2k_2})\big) \big ].
	\end{align*}
\end{definition}

We demonstrate what a progression interchange may look like in the following example. 

\begin{example}\label{ex:ProgressionInterchange}
	Consider the same schedule as in Example~\ref{ex:IllExample}. The new schedule obtained by interchanging progressions $\Pg(\ell_{31})$ and $\Pg(\ell_{52})$ is displayed in Figure~\ref{fig:Progression_Interchange}. Note that the new schedule contains unforced idleness between loops $\ell_{32}$ and $\ell_{43}$, since loop $\ell_{43}$ could actually start earlier on machine~$1$. 
	However, as $w_3+w_4=w_5=4$, the total weighted completion time is unchanged. 
\end{example}

\begin{figure*}
	\centering
	\caption{Illustration of a progression interchange}\label{fig:Progression_Interchange}
	\subfloat[Schedule before interchange]{
		\scalebox{0.85}{
			\begin{tikzpicture}[scale=0.75]
				
				\pgfmathsetmacro{\m}{3.65}
				
				\draw[-] (0,0) -- (0, -\m);
				\draw[-latex] (0, -\m) -- node[pos=1,right] {$t$} (17.5, -\m);
				
				\node (Machines) at (-1,0) {\small Machine};
				\foreach \x in {1,2,3}{
					\node (Mone) at (-1,-\x) {$\x$};
				}
				
				\foreach \x in {2, 9}{
					\node[job_gr] (1) at (\x,-1) {}; \node at (1.center) [lab2] {$1$};
					\node[job_gr] (1) at (\x+1,-2) {}; \node at (1.center) [lab2] {$1$};
					\node[job_gr] (1) at (\x+2,-3) {}; \node at (1.center) [lab2] {$1$};
				}
				\foreach \x in {3,6}{
					\node[job_gr] (2) at (\x,-1) {}; \node at (2.center) [lab2] {$2$};
					\node[job_gr] (2) at (\x+1,-2) {}; \node at (2.center) [lab2] {$2$};
					\node[job_gr] (2) at (\x+2,-3) {}; \node at (2.center) [lab2] {$2$};
				}
				\foreach \x in {1,5}{
					\node[job_gr] (4) at (\x,-1) {}; \node at (4.center) [lab2] {$4$};
					\node[job_gr] (4) at (\x+1,-2) {}; \node at (4.center) [lab2] {$4$};
					\node[job_gr] (4) at (\x+2,-3) {}; \node at (4.center) [lab2] {$4$};
				}		
				\foreach \x in {0}{
					\node[job_gr] (5) at (\x,-1) {}; \node at (5.center) [lab2] {$5$};
					\node[job_gr] (5) at (\x+1,-2) {}; \node at (5.center) [lab2] {$5$};
					\node[job_gr] (5) at (\x+2,-3) {}; \node at (5.center) [lab2] {$5$};
				}		
				
				\foreach \x in {8,11,14}{
					\node[job5] (5) at (\x,-1) {}; \node at (5.center) [lab] {$5$};
					\node[job5] (5) at (\x+1,-2) {}; \node at (5.center) [lab] {$5$};
					\node[job5] (5) at (\x+2,-3) {}; \node at (5.center) [lab] {$5$};
				}		
				
				\foreach \x in {4,7}{
					\node[job3] (3) at (\x,-1) {}; \node at (3.center) [lab] {$3$};
					\node[job3] (3) at (\x+1,-2) {}; \node at (3.center) [lab] {$3$};
					\node[job3] (3) at (\x+2,-3) {}; \node at (3.center) [lab] {$3$};
				}
				\foreach \x in {10}{
					\node[job4] (4) at (\x,-1) {}; \node at (4.center) [lab] {$4$};
					\node[job4] (4) at (\x+1,-2) {}; \node at (4.center) [lab] {$4$};
					\node[job4] (4) at (\x+2,-3) {}; \node at (4.center) [lab] {$4$};
				}

	\end{tikzpicture}}}
	
	\subfloat[Schedule after interchange]{
		\scalebox{0.85}{
			\begin{tikzpicture}[scale=0.75]
				
				\pgfmathsetmacro{\m}{3.65}
				
				\draw[-] (0,0) -- (0, -\m);
				\draw[-latex] (0, -\m) -- node[pos=1,right] {$t$} (17.5, -\m);
				
				\node (Machines) at (-1,0) {\small Machine};
				\foreach \x in {1,2,3}{
					\node (Mone) at (-1,-\x) {$\x$};
				}
				
				\foreach \x in {2, 9}{
					\node[job_gr] (1) at (\x,-1) {}; \node at (1.center) [lab2] {$1$};
					\node[job_gr] (1) at (\x+1,-2) {}; \node at (1.center) [lab2] {$1$};
					\node[job_gr] (1) at (\x+2,-3) {}; \node at (1.center) [lab2] {$1$};
				}
				\foreach \x in {3,6}{
					\node[job_gr] (2) at (\x,-1) {}; \node at (2.center) [lab2] {$2$};
					\node[job_gr] (2) at (\x+1,-2) {}; \node at (2.center) [lab2] {$2$};
					\node[job_gr] (2) at (\x+2,-3) {}; \node at (2.center) [lab2] {$2$};
				}
				\foreach \x in {1,5}{
					\node[job_gr] (4) at (\x,-1) {}; \node at (4.center) [lab2] {$4$};
					\node[job_gr] (4) at (\x+1,-2) {}; \node at (4.center) [lab2] {$4$};
					\node[job_gr] (4) at (\x+2,-3) {}; \node at (4.center) [lab2] {$4$};
				}		
				\foreach \x in {0}{
					\node[job_gr] (5) at (\x,-1) {}; \node at (5.center) [lab2] {$5$};
					\node[job_gr] (5) at (\x+1,-2) {}; \node at (5.center) [lab2] {$5$};
					\node[job_gr] (5) at (\x+2,-3) {}; \node at (5.center) [lab2] {$5$};
				}		
				\foreach \x in {8-4,11-4,14-4}{
					\node[job5] (5) at (\x,-1) {}; \node at (5.center) [lab] {$5$};
					\node[job5] (5) at (\x+1,-2) {}; \node at (5.center) [lab] {$5$};
					\node[job5] (5) at (\x+2,-3) {}; \node at (5.center) [lab] {$5$};
				}		
				
				\foreach \x in {4+4,7+4}{
					\node[job3] (3) at (\x,-1) {}; \node at (3.center) [lab] {$3$};
					\node[job3] (3) at (\x+1,-2) {}; \node at (3.center) [lab] {$3$};
					\node[job3] (3) at (\x+2,-3) {}; \node at (3.center) [lab] {$3$};
				}
				\foreach \x in {10+4}{
					\node[job4] (4) at (\x,-1) {}; \node at (4.center) [lab] {$4$};
					\node[job4] (4) at (\x+1,-2) {}; \node at (4.center) [lab] {$4$};
					\node[job4] (4) at (\x+2,-3) {}; \node at (4.center) [lab] {$4$};
				}		
				
	\end{tikzpicture}}}
\end{figure*}
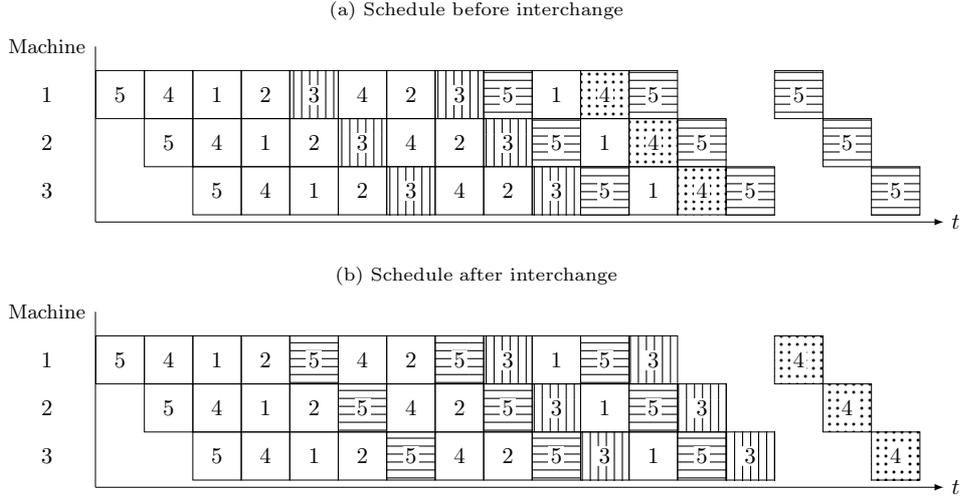


We now introduce so-called non-interruptive schedules. 
We say that a job is interrupted if there is a waiting time between the completion of a loop of that job on machine~$m$ and the start of the job's next loop on machine~$1$. 
In a non-interruptive schedule, no job has such an interruption. 

\begin{definition}
	A schedule is called non-interruptive if the next loop of any job starts on machine~$1$ as soon as its previous loop completes on machine~$m$, i.e.,
	\begin{align*}
	    S_{jk} = C_{j,k-1}, && \forall \ j=1,\ldots,n, \ k=2,\ldots,\Lps_j.
	\end{align*}

\end{definition}

Note that in a non-interruptive schedule, all loops of a job are in the same progression. 
This means that a progression starting with a loop of a particular job will finish all loops of that job before it can schedule the first loop of a different job. 
This stands in contrast to arbitrary (interruptive) schedules, where a waiting time between two consecutive loops of the same job could occur. 

The following theorem states that among the schedules with minimal total weighted completion time, there is at least one that is non-interruptive. 
This structural result implies that we can focus our analysis of total weighted completion time optimal schedules on non-interruptive schedules. 
Note that this distinguishes the total weighted completion time objective from other regular objective functions, such as the makespan, for which instances can be created where optimal schedules must contain interruptions~\citep{Yu2020}.

\begin{theorem}\label{thm:NonInterruption}
	There is a non-interruptive schedule that minimizes the total weighted completion time. 
\end{theorem}

\begin{proof} 
	The proof is by contradiction. Suppose that no non-interruptive schedule has minimal total weighted completion time. Then, an optimal schedule must have at least one interruption. Consider an optimal schedule with the least interruptions among the optimal ones.
	We consider the last interruption, i.e., the last time $t$ at which a loop $\bar{\ell} \coloneqq \ell_{jk}$ with $k<\Lps_j$ completes on machine $m$, but loop $\ell' \coloneqq \ell_{j'k'}$ with $j' \neq j$ starts on machine~$1$. 
	As $t$ is the last interruption, there are no interruptions from time $t+1$ onward. 
	Let $\ell \coloneqq \ell_{j,k+1}$ denote the next loop of job~$j$ that is scheduled on machine~$1$ at time $t+\delta > t$ and $\bar{\ell}' \coloneqq \ell_{j',k'-1}$ the previous loop of job $j'$ that is scheduled on machine~$1$ before time $t-m$. 
	Let $\Pg' \coloneqq \Pg(\ell') $ and $\Pg \coloneqq \Pg(\ell)$ denote the two progressions initiated by loops $\ell'$ and $\ell$, respectively.
	We consider three cases:	

    \pmb{(i) $\Wfl(\Pg') \leq \Wfl(\Pg)$}. We create a schedule without interruption at time $t$ by interchanging the disjoint progressions $\Pg'$ and $\Pg$, see Figure~\ref{fig:Proof_NonInterruption_Interchange_1}. The change in the objective equals ${\delta \cdot \big [\Wfl(\Pg') - \Wfl(\Pg)\big ]}$. As $\Wfl(\Pg') \leq \Wfl(\Pg)$, the total weighted completion time of this new schedule is either less than or equal to the original schedule. The new schedule also has one interruption less than the original schedule. This contradicts our assumptions.

        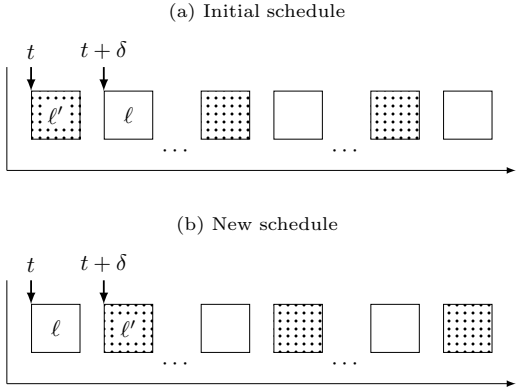
\begin{figure}
        	\centering
        	\caption{Demonstration of the progression interchange in case (i) in the proof of Theorem~\ref{thm:NonInterruption}}\label{fig:Proof_NonInterruption_Interchange_1}
        		\subfloat[Initial schedule\label{fig:Proof_NonInterruption_1A}]{
        		\scalebox{0.85}{
        		\begin{tikzpicture}[scale=0.75]
        			
        			\pgfmathsetmacro{\m}{2.15}
        			
        			\draw[-] (0.5,0) -- (0.5, -\m);
        			\draw[-latex] (0.5, -\m) -- node[pos=1,below] {} (11, -\m);
        			
        			\draw[-latex,thick] (1,0)node[above]{$t$} --(1,-0.5);
        			\draw[-latex,thick] (2.5,0)node[above]{$t+\delta$} --(2.5,-0.5);
        			
        			\foreach \x in {1}{ 
        				\node[job4] (1) at (\x,-1) {};\node at (1.center) [lab] {$\ell'$};
        			}
        			\foreach \x in {2.5}{ 
        				\node[job1] (4) at (\x,-1) {}; 
                        \node at (4.center) [lab] {$\ell$};
        			}		
        			
        			\foreach \x in {4.5,8}{ 
        				\node[job4] (41) at (\x,-1) {};
        			}
        			\foreach \x in {6,9.5}{
        				\node[job1] (41) at (\x,-1) {};
        			}
                    \foreach \x in {2.5,6}{
                        \node[job_tr] (4) at (\x+1,-1.75) {$\cdots$};
                    }
        				
        	\end{tikzpicture}}}
        	
        	\subfloat[New schedule\label{fig:Proof_NonInterruption_1B}]{
        		\scalebox{0.85}{
        		\begin{tikzpicture}[scale=0.75]
        			\pgfmathsetmacro{\m}{2.15}
        			
        			\draw[-] (0.5,0) -- (0.5, -\m);
        			\draw[-latex] (0.5, -\m) -- node[pos=1,below] {} (11, -\m);
        			
        			\draw[-latex,thick] (1,0)node[above]{$t$} --(1,-0.5);
                    \draw[-latex,thick] (2.5,0)node[above]{$t+\delta$} --(2.5,-0.5);
        			
        			\foreach \x in {2.5}{ 
        				\node[job4] (1) at (\x,-1) {};\node at (1.center) [lab] {$\ell'$};
        			}
        			\foreach \x in {1}{ 
        				\node[job1] (4) at (\x,-1) {}; \node at (4.center) [lab] {$\ell$};
        			}
        			\foreach \x in {6,9.5}{ 
        				\node[job4] (41) at (\x,-1) {};
        			}
        			\foreach \x in {4.5,8}{
        				\node[job1] (41) at (\x,-1) {};
        			}
                    \foreach \x in {2.5,6}{
                        \node[job_tr] (4) at (\x+1,-1.75) {$\cdots$};
                    }
        			
        		\end{tikzpicture}}}
        \end{figure}

        \pmb{(ii) $\Wfl(\Pg') > \Wfl(\Pg)$ and $C_{\bar{\ell'}}\leq t-m$}. We show that this is a contradiction to the original schedule being optimal. 
    	To see this, consider a new schedule that adds loop $\bar{\ell}$ as the first element to $\Pg$ and starts $\Pg'$ at time $t-m$, which corresponds to the start time of loop $\bar{\ell}$ in the old schedule, see Figure~\ref{fig:Proof_NonInterruption_Interchange_2}. 
    	As $C_{\bar{\ell'}}\leq t-m$, this new schedule is feasible. The change in the objective equals $ m \cdot \big [\Wfl(\Pg) - \Wfl(\Pg') \big ]$. Thus, the new schedule has a lower total weighted completion time. This contradicts our assumptions. 

        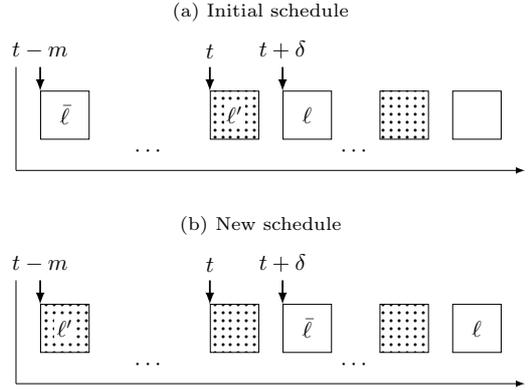
\begin{figure}
        	\centering
        	\caption{Demonstration of the progression interchange in case (ii) in the proof of Theorem~\ref{thm:NonInterruption}}\label{fig:Proof_NonInterruption_Interchange_2}
        	\subfloat[Initial schedule\label{fig:Proof_NonInterruption_2A}]{
        		\scalebox{0.85}{
        		\begin{tikzpicture}[scale=0.75]
        
        			\pgfmathsetmacro{\m}{2.15}
        			
        			\draw[-] (0.5,0) -- (0.5, -\m);
        			\draw[-latex] (0.5, -\m) -- node[pos=1,below] {} (11, -\m);
        			
        			\draw[-latex,thick] (4.5,0)node[above]{$t$} --(4.5,-0.5);
                    \draw[-latex,thick] (6,0)node[above]{$t+\delta$} --(6,-0.5);
        			\draw[-latex,thick] (1,0)node[above]{$t-m$} --(1,-0.5);
        			
        			\foreach \x in {1}{ 
        				\node[job1] (41) at (\x,-1) {$\bar{\ell}$};
        			}
        			\foreach \x in {6}{
        				\node[job1] (41) at (\x,-1) {${\ell}$};
        			}
        			\foreach \x in {9.5}{
        				\node[job1] (41) at (\x,-1) {};
        			}
        			
        			\foreach \x in {4.5}{
        				\node[job4] (4) at (\x,-1) {};\node at (4.center) [lab] {$\ell'$};
        			}
        			\foreach \x in {8}{ 
        				\node[job4] (41) at (\x,-1) {};
        			}
                    \foreach \x in {1.75,6}{
                        \node[job_tr] (4) at (\x+1,-1.75) {$\cdots$};
                    }        			
        				
        	\end{tikzpicture}}}
        	
        	\subfloat[New schedule\label{fig:Proof_NonInterruption_2B}]{
        		\scalebox{0.85}{
        		\begin{tikzpicture}[scale=0.75]
        			
        			\pgfmathsetmacro{\m}{2.15}
        			
        			\draw[-] (0.5,0) -- (0.5, -\m);
        			\draw[-latex] (0.5, -\m) -- node[pos=1,below] {} (11, -\m);
        			
        			\draw[-latex,thick] (4.5,0)node[above]{$t$} --(4.5,-0.5);
                    \draw[-latex,thick] (6,0)node[above]{$t+\delta$} --(6,-0.5);
        			\draw[-latex,thick] (1,0)node[above]{$t-m$} --(1,-0.5);
        			
        			\foreach \x in {6}{ 
        				\node[job1] (41) at (\x,-1) {$\bar{\ell}$};
        			}
        			\foreach \x in {9.5}{
        				\node[job1] (41) at (\x,-1) {${\ell}$};
        			}
        			
        			\foreach \x in {1}{
        				\node[job4] (4) at (\x,-1) {}; \node at (4.center) [lab] {$\ell'$};
        			}
        			\foreach \x in {4.5,8}{ 
        				\node[job4] (41) at (\x,-1) {};
        			}
                    \foreach \x in {1.75,6}{
                        \node[job_tr] (4) at (\x+1,-1.75) {$\cdots$};
                    }
        			
        			\end{tikzpicture}}}
        \end{figure}

        \pmb{(iii)  $\Wfl(\Pg') > \Wfl(\Pg)$ and $C_{\bar{\ell'}}> t-m$}. Similar to case (ii), we show that this is a contradiction to the original schedule being optimal. 
    	As the previous loop of job $j'$, $\bar{\ell}'$, is not yet completed at time $t-m$, the new schedule we obtained in case (ii) is infeasible. Consider the $m-1$ loops scheduled on machine~$1$ from $t-m+1$ to $t-1$. 
    	Denote these loops as $\tilde{\ell}_{m-1},\ldots,\tilde{\ell}_1$ and their corresponding progressions as $\tilde{\Pg}_{m-1},\ldots,\tilde{\Pg}_1$. 
    	Assume $\bar{\ell}'$ completes on machine $m$ at time $t-q$. 
    	This implies that loop $\ell'$ can start processing on machine~$1$ at time~$t-q$ or later. 
    	It follows that $\Wfl(\tilde{\Pg}_q)\geq\Wfl(\Pg')$ must hold; otherwise, an interchange of~$\tilde{\Pg}_q$ and~$\Pg'$ in the original schedule would have resulted in a feasible schedule with a lower total weighted completion time, which contradicts our assumptions. 

        Consider now a new schedule that adds loop $\bar{\ell}$ as the first element to $\Pg$, starts $\Pg'$ at time~${t-q}$ and $\tilde{\Pg}_q$ at time $t-m$, see Figure~\ref{fig:Proof_NonInterruption_Interchange_3}.  
    	As $\Wfl(\tilde{\Pg}_q)\geq\Wfl(\Pg')>\Wfl(\Pg)$ holds, this schedule has a lower total weighted completion time than the original schedule. 
    	If this schedule is feasible, this contradicts our assumptions. 
    	If this schedule is infeasible, one can apply a similar interchange involving $\tilde{\Pg}_q$ and one of the progressions $\tilde{\Pg}_{m-1},\ldots,\tilde{\Pg}_{q-1}$. 
    	Iteratively, this results in a feasible schedule with a lower total weighted completion time, which contradicts our assumptions. 

        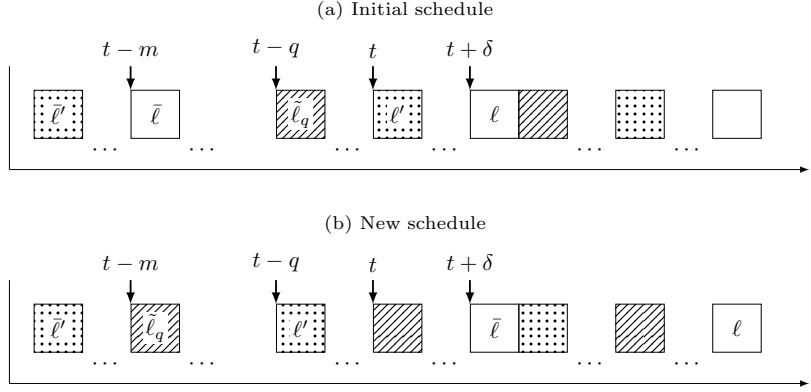
\begin{figure*}
        	\centering
        	\caption{Demonstration of the progression interchange in case (iii) in the proof of Theorem~\ref{thm:NonInterruption}}\label{fig:Proof_NonInterruption_Interchange_3}
        	\subfloat[Initial schedule\label{fig:Proof_NonInterruption_3A}]{
        		\scalebox{0.85}{
        			\begin{tikzpicture}[scale=0.75]
        				
        				\pgfmathsetmacro{\m}{2.15}
        				
        				\draw[-] (0.5,0) -- (0.5, -\m);
        				\draw[-latex] (0.5, -\m) -- node[pos=1,below] {} (17, -\m);
        				
        				\draw[-latex,thick] (8,0)node[above]{$t$} --(8,-0.5);
                        \draw[-latex,thick] (6,0)node[above]{$t-q$} --(6,-0.5);
                        \draw[-latex,thick] (10,0)node[above]{$t+\delta$} --(10,-0.5);
        				\draw[-latex,thick] (3,0)node[above]{$t-m$} --(3,-0.5);
        				
        				\foreach \x in {1}{ 
        						\node[job4] (4) at (\x,-1) {};\node at (4.center) [lab] {$\bar{\ell}'$};
        				}
        				\foreach \x in {3}{ 
        					\node[job1] (41) at (\x,-1) {$\bar{\ell}$};
        				}
        				\foreach \x in {6}{ 
        					\node[job2] (2) at (\x,-1) {};\node at (2.center) [lab] {$\tilde{\ell}_q$};
        				}
        				\foreach \x in {8}{ 
        					\node[job4] (4) at (\x,-1) {};\node at (4.center) [lab] {${\ell}'$};
        				}
        				\foreach \x in {10}{ 
        					\node[job1] (1) at (\x,-1) {};\node at (1.center) [lab] {${\ell}$};
        				}
        				\foreach \x in {13}{ 
        					\node[job4] (4) at (\x,-1) {};
        				}
        				\foreach \x in {15}{ 
        					\node[job1] (1) at (\x,-1) {};
        				}
        				\foreach \x in {11}{ 
        					\node[job2] (4) at (\x,-1) {};
        				}
                        \foreach \x in {1,3,6,8,11,13}{
                            \node[job_tr] (41) at (\x+1,-1.75) {$\cdots$};
                        }
        				
        	\end{tikzpicture}}}
        	
        	\subfloat[New schedule\label{fig:Proof_NonInterruption_3B}]{
        		\scalebox{0.85}{
        			\begin{tikzpicture}[scale=0.75]
        				
        				\pgfmathsetmacro{\m}{2.15}
        				
        				\draw[-] (0.5,0) -- (0.5, -\m);
        				\draw[-latex] (0.5, -\m) -- node[pos=1,below] {} (17, -\m);
        				
        				\draw[-latex,thick] (8,0)node[above]{$t$} --(8,-0.5);
                        \draw[-latex,thick] (6,0)node[above]{$t-q$} --(6,-0.5);
                        \draw[-latex,thick] (10,0)node[above]{$t+\delta$} --(10,-0.5);
        				\draw[-latex,thick] (3,0)node[above]{$t-m$} --(3,-0.5);
        				
        				\foreach \x in {1}{ 
        					\node[job4] (4) at (\x,-1) {};\node at (4.center) [lab] {$\bar{\ell}'$};
        				}
        				\foreach \x in {10}{ 
        					\node[job1] (41) at (\x,-1) {$\bar{\ell}$};
        				}
        				\foreach \x in {3}{ 
        					\node[job2] (2) at (\x,-1) {};\node at (2.center) [lab] {$\tilde{\ell}_q$};
        				}
        				\foreach \x in {8}{ 
        					\node[job2] (2) at (\x,-1) {};
        				}
        				\foreach \x in {6}{ 
        					\node[job4] (4) at (\x,-1) {}; \node at (4.center) [lab] {${\ell}'$};
        				}
        				\foreach \x in {11}{ 
        					\node[job4] (4) at (\x,-1) {}; 
        				}
        				\foreach \x in {15}{ 
        					\node[job1] (1) at (\x,-1) {};\node at (1.center) [lab] {${\ell}$};
        				}
        				\foreach \x in {13}{ 
        					\node[job2] (4) at (\x,-1) {};
        				}
                        \foreach \x in {1,3,6,8,11,13}{
                            \node[job_tr] (41) at (\x+1,-1.75) {$\cdots$};
                        }
        				
        	\end{tikzpicture}}}
        \end{figure*}
\end{proof}


\section{Complexity and priority rules}\label{sec:m_input} 


In this section, we analyze the complexity of ($F|reentry, \ p_{ijk} = 1|\sum w_jC_j$) and present two priority rules. 
In Subsection~\ref{subsec:NP_hardness}, we show that the problem is strongly NP-hard by a reduction from 3-PARTITION. 
In Subsection~\ref{subsec:LRL}, we present the LRL priority rule and show that it minimizes the total completion time as well as the total weighted completion time if the weights are agreeable.
In Subsection~\ref{subsec:WLRL}, we present the WLRL priority rule and show its worst-case performance ratio. 

\subsection{Complexity}\label{subsec:NP_hardness}

The decision problem 3-PARTITION is a well-known strongly NP-hard decision problem.
\begin{definition}
	The decision problem 3-PARTITION is defined as follows. Given integers~${a_1, \ldots, a_{3q},b,q}$, such that $\frac{b}{4}<a_j<\frac{b}{2}$ and $qb = \sum_{j=1}^{3q} a_j$, is there a partition of the first $3q$ integers into~$q$ disjoint triplets $\St_1, \ldots, \St_q$ with $\sum_{a_j\in S_t } a_j = b$ for all $t = 1,\ldots,q$?
\end{definition}

Before we present the reduction, we provide another formulation of the objective function. We denote the progression that starts at times $t = 0,1,\ldots,m-1$ by~$\Pg_{t}$.
 Next, we rewrite the objective for the case $w_j = \Lps_j$ for all jobs $j$. The result is similar to \cite[page 270]{Eastman1964} for a parallel machine setting.
\begin{lemma}
    \label{lemma: Formula for the Objective}
    Given m machines and weights $w_j = \Lps_j$, we obtain for any non-interruptive schedule without intermediate idle times:
    \begin{alignat*}{3}
        \sum_{j=1}^{n} \Lps_j C_j = & &&\frac{m}{2} \cdot \sum_{t=1}^{m} \left(\lvert\Pg_t\rvert^2 + \lvert\Pg_t\rvert \frac{2 \cdot (t-1)}{m}\right) && \\
        & &+ &\frac{m}{2}\cdot \sum_{j=1}^{n} \Lps_j^2. &&
    \end{alignat*}
\end{lemma}
\begin{proof}
    Let $1,\ldots,n_{t}$ be the jobs assigned to progression $\Pg_t$. The objective for these jobs can be rewritten:
    \begin{alignat*}{3}
        \sum_{j=1}^{n_t} \Lps_j C_j & = && 
        \sum_{j=1}^{n_t} \Lps_j \left( t-1 + \sum_{k=1}^j \Lps_k \cdot m \right) && \\
        & = && \frac{m}{2} \cdot\left( \lvert\Pg_t\rvert^2 + \sum_{j=1}^{n_t} \Lps_j^2\right) && \\
        & & + &\lvert\Pg_t\rvert\cdot\left( t-1\right). &&
    \end{alignat*}
    Summing the terms for all progressions $\Pg_1,\ldots,\Pg_t$ yields the claimed formula.
\end{proof}
The above calculation also shows that the objective is independent of the ordering of jobs within a progression if $w_j=\Lps_j$ holds for all jobs. Using Lemma~\ref{lemma: Formula for the Objective}, it can be shown that it is optimal to balance the number of loops in the progressions as much as possible.




\begin{theorem}
    The scheduling problem $F|reentry, \ p_{ijk} = 1|\sum w_jC_j$ is strongly NP-hard.
\end{theorem}

\begin{proof}
    Let $(b,q,a_1,\ldots, a_{3q})$ be an instance of 3-PARTITION. For each item $a_k$, $k=1,\ldots, 3q$, we define a job $k$ with $w_{k} = \Lps_{k} = a_k$. Further, we choose the number of machines to be equal to the number of triplets, i.e., $m = q$. By Lemma~\ref{lemma: Formula for the Objective}, the cost of a non-interruptive schedule without intermediate idle times on this instance is given as $\frac{m}{2} \left( \sum_{t = 1}^m \left(\lvert\Pg_t\rvert^2 + \lvert\Pg_t\rvert\frac{2 \cdot (t-1)}{m}\right) +\sum_{k=1}^{3q} \Lps_{k}^2\right)$. Notice that the last term does not depend on the assignment of jobs to progressions. We show that this scheduling instance has an objective value of at most $\frac{m}{2} \left( mb^2 + (m-1) b +  \sum_{k=1}^{3q} \Lps_k^2 \right)$ if and only if $(b,q,a_1,\ldots, a_{3q})$ is a YES-instance of 3-PARTITION.

    Let the instance be a YES-instance of 3-PARTITION and the jobs assigned into triplets $S_1, S_2, \ldots, S_q$ such that $\sum_{a_j \in S_t} a_j = b$ for all $t=1,\ldots,q$. We can assign the jobs from $S_t$ to progression $\Pg_t$ to get the above objective value.

    Reversely, assume there exists a schedule of value at most the above value. If ${\lvert\Pg_1\rvert = \lvert\Pg_2\rvert = \ldots = \lvert\Pg_m\rvert}$, then we are done. Otherwise, two progressions~$\Pg_i, \Pg_l$ must exist such that $\lvert\Pg_i\rvert \geq \lvert\Pg_l\rvert + 2$. Since the numbers of loops in the progressions are not perfectly balanced, we can use $\sum_{t = 1}^m \lvert\Pg_t\rvert = q \cdot b$ to show $\sum_{t = 1}^m \left(\lvert\Pg_t\rvert^2 + \lvert\Pg_t\rvert\frac{2 \cdot (t-1)}{m}\right) > mb^2 + (m-1) b$. 
    This result contradicts our assumption, which concludes the proof.
\end{proof}


\subsection{Least Remaining Loops priority rule}\label{subsec:LRL}

In this subsection, we introduce the ``Least Remaining Loops first'' (LRL) priority rule.
We show that this rule minimizes the total (unweighted) completion time, i.e., $w_j=1$ for all jobs $j=1,\ldots,n$. 
Also, we show that it minimizes the total weighted completion time if the weights are agreeable. 

\begin{definition}\label{def:LRL}
	The LRL priority rule selects, whenever machine~$1$ becomes available, the next loop of a job with the least remaining loops among all available jobs. 
\end{definition}

The following example displays an LRL schedule. 

\begin{example}\label{ex:IllExample_LRL}
	Consider the same instance as in Example~\ref{ex:IllExample} with ${n =5}$ and ${w_j = 1}$ for all $j = 1,\ldots,5$.
	The LRL rule generates the permutation sequence $\Sched = [\ell_{11};\allowbreak\ell_{21};\allowbreak\ell_{31};\allowbreak\ell_{12};\allowbreak\ell_{22};\allowbreak\ell_{32};\allowbreak\ell_{41};\allowbreak\ell_{51};\allowbreak\ell_{42};\allowbreak\ell_{52};\allowbreak\ell_{43};\allowbreak\ell_{53};\allowbreak\ell_{54}]$. The Gantt chart for this schedule is shown in Figure~\ref{fig:ExampleGanttChart_LRL}. 
	The total completion time of this schedule is $C_1+C_2+C_3+C_4+C_5= 55$. 
\end{example}

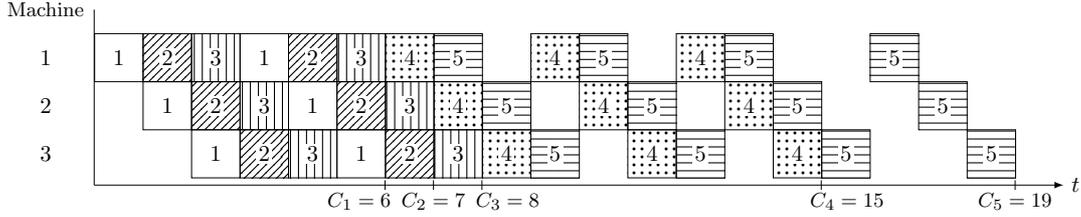
\begin{figure*}
	\centering
	\caption{Schedule generated via priority rule LRL}\label{fig:ExampleGanttChart_LRL}
	\scalebox{0.85}{
	\begin{tikzpicture}[scale=0.75]
		
		\tikzset{job/.style={rectangle,draw,anchor=west,minimum height=0.75cm,minimum width=0.75cm}}				
		\pgfmathsetmacro{\m}{3.65}
		
		\draw[-] (0,0) -- (0, -\m);
		\draw[-latex] (0, -\m) -- node[pos=1,right] {$t$} (20, -\m);
		
		
		\node (Machines) at (-1,0) {\small Machine};
		\foreach \x in {1,2,3}{
			\node (Mone) at (-1,-\x) {$\x$};
		}
		
		\foreach \x in {0,3}{
			\node[job1] (1) at (\x,-1) {}; \node at (1.center) [lab] {$1$};
			\node[job1] (1) at (\x+1,-2) {}; \node at (1.center) [lab] {$1$};
			\node[job1] (1) at (\x+2,-3) {}; \node at (1.center) [lab] {$1$};
		}
		\foreach \x in {1,4}{
			\node[job2] (2) at (\x,-1) {}; \node at (2.center) [lab] {$2$};
			\node[job2] (2) at (\x+1,-2) {}; \node at (2.center) [lab] {$2$};
			\node[job2] (2) at (\x+2,-3) {}; \node at (2.center) [lab] {$2$};
		}
		\foreach \x in {2,5}{
			\node[job3] (3) at (\x,-1) {}; \node at (3.center) [lab] {$3$};
			\node[job3] (3) at (\x+1,-2) {}; \node at (3.center) [lab] {$3$};
			\node[job3] (3) at (\x+2,-3) {}; \node at (3.center) [lab] {$3$};
		}
		\foreach \x in {6,9,12}{
			\node[job4] (4) at (\x,-1) {}; \node at (4.center) [lab] {$4$};
			\node[job4] (4) at (\x+1,-2) {}; \node at (4.center) [lab] {$4$};
			\node[job4] (4) at (\x+2,-3) {}; \node at (4.center) [lab] {$4$};
		}		
		\foreach \x in {7,10,13,16}{
			\node[job5] (5) at (\x,-1) {}; \node at (5.center) [lab] {$5$};
			\node[job5] (5) at (\x+1,-2) {}; \node at (5.center) [lab] {$5$};
			\node[job5] (5) at (\x+2,-3) {}; \node at (5.center) [lab] {$5$};
		}
		
		\foreach \x in {6}{
			\node[anchor=north,xshift=-4mm] (C4) at (\x,-\m) {\small $C_1=\x$};
			\draw (\x,-0.1-\m) -- (\x,0.1-\m);
		}
		\foreach \x in {7}{
			\node[anchor=north,xshift=0mm] (C4) at (\x,-\m) {\small $C_2=\x$};
			\draw (\x,-0.1-\m) -- (\x,0.1-\m);
		}
		\foreach \x in {8}{
			\node[anchor=north,xshift=4mm] (C4) at (\x,-\m) {\small $C_3=\x$};
			\draw (\x,-0.1-\m) -- (\x,0.1-\m);
		}
		\foreach \x in {15}{
			\node[anchor=north,xshift=4mm] (C4) at (\x,-\m) {\small $C_4=\x$};
			\draw (\x,-0.1-\m) -- (\x,0.1-\m);
		}
		\foreach \x in {19}{
			\node[anchor=north] (C4) at (\x,-\m) {\small $C_5=\x$};
			\draw (\x,-0.1-\m) -- (\x,0.1-\m);
		}

	\end{tikzpicture}}
\end{figure*}

After a loop of a job has been scheduled, that job will have one less loop remaining the next time it is available to be processed on machine~$1$. 
Thus, the job's priority is even higher than before. 
This implies that the next loop of that job will be scheduled immediately without interruption.

\begin{lemma}
	Schedules generated via the LRL rule are non-interruptive.
\end{lemma}
\begin{proof}
	If a loop of a job is scheduled at time $t$, it has priority over all jobs that are available at time $t$. The next time a loop of this job is available, it has one less loop than before. One can inductively show that this job is the unique one with the lowest number of remaining loops left among all jobs available at time $t+m$. Therefore, it again has priority at $t+m$ according to the LRL rule.
\end{proof}

We now show that LRL minimizes the total completion time. The proof uses a similar interchange argument as in Theorem~\ref{thm:NonInterruption}. 
However, it focuses on the first interruption rather than the last. 

\begin{theorem}\label{thm:LRL_MinTCT}
	LRL minimizes the total completion time. 
\end{theorem}

\begin{proof}
	The proof is by contradiction. Suppose LRL does not minimize the total completion time. 
	This means that every optimal schedule deviates from LRL at least once. 
	According to Theorem~\ref{thm:NonInterruption}, we can assume that this optimal schedule has no interruptions. 	
	Let time $t$ be the first time the optimal schedule does not act according to LRL with loop $\ell'\coloneqq \ell_{j'k'}$ being scheduled while loop $\ell\coloneqq \ell_{jk}$ would have had priority under LRL. Since the schedule has no interruptions, both $\ell$ and $\ell'$ must be the first loop of their respective jobs $j$ and $j'$. Define $\Pg' \coloneqq \Pg(\ell')$ and $\Pg \coloneqq \Pg(\ell)$. Note that $\Wfl(\Pg)$ is the number of jobs in this progression, since $w_j=1$ for all jobs~$j$.
	We distinguish two cases: 

    \pmb{(i) $\Wfl(\Pg') \leq \Wfl(\Pg)$}. We create a schedule that acts according to LRL at time $t$ by interchanging the disjoint progressions $\Pg'$ and $\Pg$. As $\Wfl(\Pg') \leq \Wfl(\Pg)$, the total weighted completion time of this new schedule is either less than or equal to the original schedule. This contradicts our assumptions. 
	
    \pmb{(ii) $\Wfl(\Pg') > \Wfl(\Pg)$}. We create a schedule that acts according to LRL at time $t$ by interchanging all loops of jobs $j'$ and $j$.
    As $\Lps_j<\Lps_{j'}$, all subsequent jobs in $\Pg'$ have an earlier completion time, while all subsequent jobs in $\Pg$ have a later completion time. 
	As $\Wfl(\Pg') > \Wfl(\Pg)$, the total completion time of the new schedule is lower. This contradicts our assumptions. Note that this interchange also works if the progressions are not disjoint, i.e., $\Pg\subset \Pg'$.
	
\end{proof}

The previous theorem implies that any non-interruptive schedule minimizes the total completion time in case all jobs have the same number of loops, i.e., $\Lps_{j} = \Lps$ for all $j=1,\ldots,n$. 
The class of non-interruptive schedules can be seen as a ``counterpart'' to so-called ``Loopwise-Cyclic'' schedules that minimize the makespan for such instances \citep{Yu2020}. 

We now consider instances in which the weights may differ but are ``agreeable'' with the number of loops of the job.

\begin{definition}\label{def:Agreeable_Weights}
    The weights and the numbers of loops are ``agreeable'' if $\Lps_j \leq \Lps_{j'} \Leftrightarrow w_j \geq w_{j'}$ holds for any pair of jobs~$j$ and~$j'$, and at least one of the inequalities is strict.    
\end{definition}

The next lemma shows that if the weights and number of loops are agreeable, LRL minimizes the total weighted completion time. 

\begin{lemma}\label{thm:LRL_Agreeable}
    LRL minimizes the total weighted completion time if the weights and number of loops are agreeable and ties are broken by the highest weight. 
\end{lemma}
\begin{proof}
    The proof is similar to the proof of Theorem~\ref{thm:LRL_MinTCT}. Let time $t$ be the first time the optimal schedule does not act according to LRL with loop $\ell'\coloneqq \ell_{j'k'}$ being scheduled while loop $\ell\coloneqq \ell_{jk}$ would have had priority under LRL. Consider the same two cases:

    \pmb{(i) $\Wfl(\Pg') \leq \Wfl(\Pg)$}. Interchange the disjoint progressions $\Pg'$ and $\Pg$ (see proof of Theorem~\ref{thm:LRL_MinTCT}).
	
    \pmb{(ii) $\Wfl(\Pg') > \Wfl(\Pg)$}. Interchange all loops of jobs $j'$ and $j$. 
    As $\Lps_j\leq \Lps_{j'}$, all subsequent jobs in $\Pg'$ have a completion time that is not later than in the original schedule, while all subsequent jobs in $\Pg$ have a completion time that is not earlier. 
	As $\Wfl(\Pg') > \Wfl(\Pg)$ and $w_j \geq w_{j'}$, the total weighted completion time of the new schedule cannot be higher. As one of the inequalities $\Lps_j\leq \Lps_{j'}$ or $w_j \geq w_{j'}$ is strict, the total weighted completion time is strictly lower, which contradicts our assumptions. 
\end{proof}

\subsection{Weighted Least Remaining Loops priority rule}\label{subsec:WLRL}
In this subsection, we introduce the ``Weighted Least Remaining Loops first'' (WLRL) rule. We show that it produces schedules whose total weighted completion time exceeds the optimal total weighted completion by a factor of at most $\frac{1}{2}\big(1+\sqrt{2}\big) \approx 1.207$. 




\begin{definition}\label{def:WLRL}
	The WLRL priority rule selects, whenever machine~$1$ becomes available, the next loop of a job with the highest $\frac{w_j}{\Lps_j}$ ratio among all available jobs. 

\end{definition}

The following example displays a WLRL schedule and shows that the rule may not generate an optimal solution. 

\begin{example}\label{ex:IllExample_WLRL}
	Consider a reentrant flow shop with $m=2$ machines and $n = 3$ jobs. The jobs require $\Lps_1=2,\ \Lps_2=2$, and $\Lps_3=6$ loops, and have weights $w_1=2.2, \ w_2=2.1$, and $w_3 = 6$. 
	The WLRL rule generates the permutation sequence $\Sched = [\ell_{11};\ell_{21};\ell_{12};\ell_{22};\ell_{31};\ell_{32};\ell_{33};\ell_{34};\ell_{35};\ell_{36}]$. The Gantt chart for this schedule is displayed in Figure~\ref{fig:ExampleGanttChart_WLRL_A}. 
	The total weighted completion time of this schedule is $115.3$.
	
	The optimal schedule $\tilde{\Sched}$ for this instance is depicted in Figure~\ref{fig:ExampleGanttChart_WLRL_B} and has a total weighted completion time of $101.9$.
\end{example}

\begin{figure*}
	\centering
	\caption{Gantt charts for the WLRL schedule and an optimal schedule of Example~\ref{ex:IllExample_WLRL}}\label{fig:ExampleGanttChart_WLRL}
	\subfloat[WLRL schedule $\Sched$]{
		\label{fig:ExampleGanttChart_WLRL_A}
		\scalebox{0.85}{
		\begin{tikzpicture}[scale=0.75]
			\tikzset{job/.style={rectangle,draw,anchor=west,minimum height=0.75cm,minimum width=0.75cm}}			
			\pgfmathsetmacro{\m}{2.65}
			
			\draw[-] (0,0) -- (0, -\m);
			\draw[-latex] (0, -\m) -- node[pos=1,right] {$t$} (17, -\m);
			
			\node (Machines) at (-1,0) {\small Machine};
			\foreach \x in {1,2}{
				\node (Mone) at (-1,-\x) {$\x$};
			}
			
			\draw[-,thick] (4,-\m+0.1) -- node[below,xshift=-5mm] {$C_1(\Sched)=4$} (4, -\m-0.1);
			\draw[-,thick] (5,-\m+0.1) -- node[below,xshift=5mm] {$C_2(\Sched)=5$} (5, -\m-0.1);
			\draw[-,thick] (16,-\m+0.1) -- node[below] {$C_3(\Sched)=16$} (16, -\m-0.1);
			
			\foreach \x in {0,2}{
				\node[job1] (1) at (\x,-1) {}; \node at (1.center) [lab] {$1$};
				\node[job1] (1) at (\x+1,-2) {}; \node at (1.center) [lab] {$1$};
			}
			\foreach \x in {1,3}{
				\node[job2] (2) at (\x,-1) {}; \node at (2.center) [lab] {$2$};
				\node[job2] (2) at (\x+1,-2) {};  \node at (2.center) [lab] {$2$};
			}
			\foreach \x in {4,6,8,10,12,14}{
				\node[job4] (4) at (\x,-1) {};\node at (4.center) [lab] {$3$};
				\node[job4] (4) at (\x+1,-2) {};\node at (4.center) [lab] {$3$};
			}
			\end{tikzpicture}}}
		
	\subfloat[Optimal schedule $\tilde{\Sched}$]{
		\label{fig:ExampleGanttChart_WLRL_B}
		\scalebox{0.85}{
		\begin{tikzpicture}[scale=0.75]
			\tikzset{job/.style={rectangle,draw,anchor=west,minimum height=0.75cm,minimum width=0.75cm}}			
			\pgfmathsetmacro{\m}{2.65}
			
			\draw[-] (0,0) -- (0, -\m);
			\draw[-latex] (0, -\m) -- node[pos=1,right] {$t$} (17, -\m);
			
			\node (Machines) at (-1,0) {\small Machine};
			\foreach \x in {1,2}{
				\node (Mone) at (-1,-\x) {$\x$};
			}
			
			\draw[-,thick] (5,-\m+0.1) -- node[below,xshift=-0mm] {$C_1(\tilde{\Sched})=5$} (5, -\m-0.1);
			\draw[-,thick] (9,-\m+0.1) -- node[below,xshift=0mm] {$C_2(\tilde{\Sched})=9$} (9, -\m-0.1);
			\draw[-,thick] (12,-\m+0.1) -- node[below] {$C_3(\tilde{\Sched})=12$} (12, -\m-0.1);
			
			\foreach \x in {1,3}{
				\node[job1] (1) at (\x,-1) {}; \node at (1.center) [lab] {$1$};
				\node[job1] (1) at (\x+1,-2) {}; \node at (1.center) [lab] {$1$};
			}
			\foreach \x in {5,7}{
				\node[job2] (2) at (\x,-1) {}; \node at (2.center) [lab] {$2$};
				\node[job2] (2) at (\x+1,-2) {};  \node at (2.center) [lab] {$2$};
			}
			\foreach \x in {0,2,4,6,8,10}{
				\node[job4] (4) at (\x,-1) {};\node at (4.center) [lab] {$3$};
				\node[job4] (4) at (\x+1,-2) {};\node at (4.center) [lab] {$3$};
			}
			
	\end{tikzpicture}}}
\end{figure*}
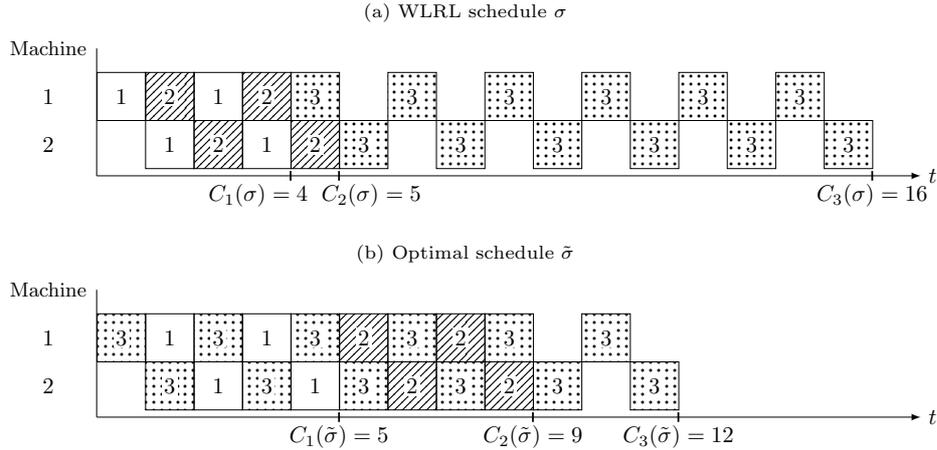


Although it is not guaranteed that the WLRL rule generates an optimal schedule, it seems to work reasonably well in practice. One can measure the performance of an algorithm by analyzing the factor by which the algorithm's solution is larger than the optimal solution. 
This factor is also called ``approximation ratio'' or ``performance ratio''. Let $\OFVV(ALG)$ and $\OFVV(OPT)$ describe the objective function values of the algorithm's solution and the optimal solution, respectively. 
The algorithm's performance ratio is $\frac{\OFVV(ALG)}{\OFVV(OPT)}$. 
In a numerical experiment involving 20,000 randomly generated instances with $4$ to $8$ jobs, $2$ to $6$ machines, $1$ to $20$ loops per job, and weights ranging from $1$ to $20$ (all sampled uniformly random), the average performance ratio of the WLRL rule was $1.01$, while the empirical worst-case performance ratio was $1.11$. 

Our goal in what follows is to derive and prove a theoretical worst-case bound on the performance ratio of the WLRL rule. 
For the parallel machine scheduling problem $P||\sum w_jC_j$, \cite{Kawaguchi1986} showed that the ``Weighted Shortest Processing Time First'' (WSPT) rule, which schedules the jobs following the $\frac{w_j}{p_j}$ priority rule~(see,~e.g.,~\citealp{Pinedo2022}), has a worst-case performance ratio of $\frac{1}{2}\big(1+\sqrt{2}\big)$. 
\cite{Schwiegelshohn2011} presented an alternative proof of this bound by analyzing the structure of worst-case instances. 
We show that for reentrant flow shops with unit processing times, the WLRL rule also has a worst-case performance ratio of $\frac{1}{2}\big(1+\sqrt{2}\big)$. 
We achieve our result in three steps.
First, we identify properties of a worst-case instance, similar to \cite{Schwiegelshohn2011}. 
For instances that satisfy these properties, we then create related instances of the $P||\sum w_jC_j$ problem in the second step. 
Finally, we use these instances to conclude our proof. 

Throughout this section, we denote an instance of the reentrant flow shop problem as $I$. 
The WLRL schedule and optimal schedule are denoted as $WLRL(I)$ and $OPT(I)$, respectively, or $WLRL$ and $OPT$ if the instance is clear from the context.
Furthermore, we denote the performance ratio of the WLRL rule as:

\begin{align*}
    \Rto(I) = \frac{\OFVV\big(WLRL(I)\big)}{\OFVV\big(OPT(I)\big)} = \frac{\sum w_jC_j\big(WLRL(I)\big)}{\sum w_jC_j\big(OPT(I)\big)}.
\end{align*} 

The first step of the proof is to identify two properties that potential worst-case instances must satisfy. 
The following lemma shows that in a worst-case instance, the weight of each job is equal to its number of loops. 

\begin{lemma}\label{lem:WeightEqualLoops}
	For every instance $I$, there is an instance $I'$ with $\Rto(I) \leq \Rto(I')$ and $w_j=\Lps_j$ for all jobs $j$.
\end{lemma}
\begin{proof}
	Replace $p_j$ by $\Lps_{j}$ in Corollary~1 of~\cite{Schwiegelshohn2011}. 
\end{proof}

If $w_j = \Lps_j$ for all jobs~$j$, any list schedule can be generated according to the WLRL rule by using appropriate tie-breakers. 
In fact, we can create any list schedule via the WLRL rule by adding or subtracting arbitrarily small $\delta$ values to the weights.
The influence of such perturbations on the performance ratio is negligible.
Therefore, we can focus our analysis on a worst-case and an optimal order. 

The following lemma shows that a potential worst-case instance has some ``small'' jobs with one loop and at most $m-1$ ``big'' jobs with two or more loops. 
Let $T_{idl}(\Sched)$ denote the first idle time on machine~$1$ under schedule~$\Sched$. 
The idea is to create another instance with a higher performance ratio that splits off one loop of a job that currently completes earlier than the first idle time on machine 1 in the worst-case schedule. 

\begin{lemma}\label{lem:SmallAndBigJobs}
	For every instance $I$, there is an instance $I'$ with $\Rto(I)\leq \Rto(I')$, $w_j = \Lps_j$ for all jobs~$j$, and $\Lps_j=1$ if $S_j < T_{idl}\big(WLRL(I')\big)-m$. 
\end{lemma}

\begin{proof}
	By Lemma~\ref{lem:WeightEqualLoops}, we can assume that $w_j = \Lps_j$ for all jobs~$j$ in $I$.
    Let $\Sched(I)$ be an arbitrary schedule and let $j_0$ be a job from instance~$I$ with $\Lps_{j_0} > 1$ starting before time~$T_{idl}\big(WLRL(I)\big)-m$ in $\Sched(I)$. We split this job into two smaller jobs $j_1$ and $j_2$ with $\Lps_{j_1}=w_{j_1} = 1$ and $\Lps_{j_2}=w_{j_2}=\Lps_{j_0}-1$ to produce a new instance $I'$.
    We extend $\Sched(I)$ to $\Sched(I')$ by scheduling $S_{j_1}\big(\Sched(I')\big)=S_{j_0}\big(\Sched(I)\big)$, $C_{j_1}\big(\Sched(I')\big) = S_{j_2}\big(\Sched(I')\big)=S_{j_0}+m$, which implies $C_{j_2}\big(\Sched(I')\big) = C_{j_0}\big(\Sched(I)\big)$. 
    This operation clearly does not affect the completion time of any other job and, by the assumption $S_{j_0} + m <T_{idl}\big(WLRL(I')\big)$, also results in a feasible WLRL schedule.
    
	Such a split (strictly) decreases the total weighted completion time by:
	\begin{multline*}
		\Lps_{j_0}C_{j_0} - (S_{j_0}+m) - (\Lps_{j_0}-1)C_{j_0} \\
        = C_{j_0} - S_{j_0} - m = m \left( \Lps_{j_0} - 1 \right ) > 0.
	\end{multline*}
	It follows that:
	\begin{align*}
		1 \leq \Rto(I)  & \leq \frac{\sum w_jC_j\big(WLRL(I)\big) - m \left( \Lps_{j_0} - 1 \right )}{\sum w_jC_j\big(OPT(I)\big) - m \left( \Lps_{j_0} - 1 \right )} \\
        & \leq \frac{\sum w_jC_j\big(WLRL(I')\big)}{\sum w_jC_j\big(OPT(I')\big)} = \Rto(I').
	\end{align*}
    Iteratively, splitting jobs yields an instance $I'$ with the claimed properties.
	%
\end{proof}

Using this lemma, we can focus on instances that contain $x$ small jobs with one loop and $B<m$ big jobs with $2 \leq y_0 \leq \ldots \leq y_{B-1}$ loops, all satisfying $w_j = \Lps_j$. 
One can show, using an interchange argument, that an optimal order for such an instance is generated by breaking ties according to the highest weight. That is, starting the job with the most loops at time $0$, the job with the second most loops at time $1$, and so on. 
Conversely, one can show that a worst-case order can be generated by breaking ties according to the smallest weight. That is, starting the $x$ small jobs at times $0,1,\ldots,x-1$ and the big jobs at times $x,x+1,\ldots,x+B-1$. 

The second step of the proof is to transform a potential worst-case instance $I$ into a corresponding instance $\bar{I}$ of the $P||\sum w_jC_j$ problem. 
We create such an instance that contains all jobs of instance $I$ (``real'' jobs) plus some ``dummy'' jobs. 
Then, we split a subset of the real jobs of this instance to offset the total weighted completion time of the dummy jobs.  

\begin{definition}
	For an instance $I$ as described above, we create an instance $\bar{I}$ of the $P||\sum w_jC_j$ problem as follows: $\bar{p}_j \coloneqq m\Lps_{j}, \ \bar{w}_j \coloneqq m\Lps_{j},  \ j=1,\ldots,x+B$. 
	Additionally, we define $m-1$ dummy jobs with $\bar{p}_{x+B+i} \coloneqq i, \ \bar{w}_{x+B+i} \coloneqq  i, \ i = 1,\ldots,m-1$. 
\end{definition}

As $\bar{w}_j = \bar{p}_j $ holds for all jobs in instance $\bar{I}$, any list schedule can be seen as a WSPT schedule; we can achieve a unique WSPT order by modifying the weights by very small values $\delta$. 

Next, we show that for a convenient ordering, the ``real'' jobs of instance $\bar{I}$ have the same completion times as in a worst-case schedule of instance $I$.  

\begin{lemma}\label{lem:CompletionTime_I_bar}
	There is a WSPT ordering of the jobs of instance $\bar{I}$ such that the completion times of the real jobs are equal to the completion times of a worst-case WLRL schedule of instance $I$.
\end{lemma}
\begin{proof}
	We schedule the jobs as follows: 
	schedule real job~$1$ at time~$0$ on machine~$1$. 
	Schedule the dummy jobs such that dummy job $x+B+i-1$ with length $i-1$ is scheduled on machine~$i$, $i = 2,\ldots, m$. 
	Schedule real job~$2$ at time~$1$ on machine~$2$. 
	Schedule real job~$3$ at time~$2$ on machine~$3$ and so on, see Figure~\ref{fig:Schematic_Display_I_bar}. 
	%
	One can easily verify that the completion times of this schedule are equal to those of the worst-case schedule of instance~$I$.
\end{proof}

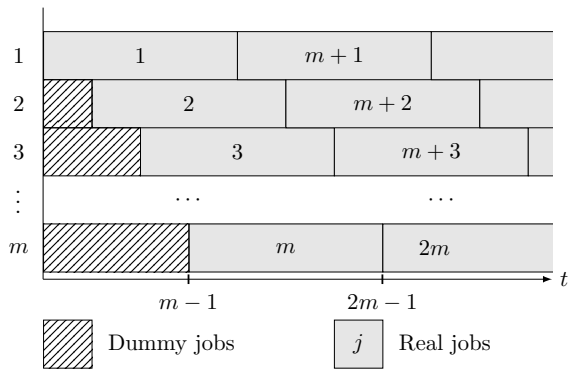
\begin{figure}
	\centering
	\caption{Schematic Gantt chart for the worst-case WSPT schedule in the proof of Lemma~\ref{lem:CompletionTime_I_bar}}\label{fig:Schematic_Display_I_bar}
		\scalebox{0.85}{
			\begin{tikzpicture}[scale=0.75]
				\tikzset{job_m/.style={rectangle,draw,anchor=west,minimum height=0.75cm,minimum width=4*.75cm}}			
				\tikzset{job_d/.style={rectangle,draw,anchor=west,minimum height=0.75cm}}			
				\pgfmathsetmacro{\m}{5.65}

				\draw[-] (0,0) -- (0, -\m);
				\draw[-latex] (0, -\m) -- node[pos=1,right] {$t$} (10.5, -\m);
				\foreach \x in {1,2,3}{
					\node (Mone) at (-0.5,-\x) {$\x$};
				}
				\node[rotate=90] (Mone) at (-0.5,-4) {$\cdots$};
				\node (Mone) at (-0.5,-5) {$m$};
																
				\foreach \x in {1,2,3}{
					\node[job_m,fill=cj1] (41) at (\x-1,-\x) {$\x$};
				}
				\node[rectangle,anchor=west,minimum height=0.75cm,minimum width=0.75] (Mone) at (2.55,-4) {$\cdots$};
				\node[job_m,fill=cj1] (41) at (4-1,-5) {$m$};
				
				\begin{scope}
                    \clip (0,0) rectangle (10.5,-6.2); 
                    \foreach \x in {1,2,3}{
                        \node[job_m,fill=cj1] (41) at (4+\x-1,-\x) {$m+\x$};
                    }
                    \node (Mone) at (5+4-.75,-4) {$\cdots$};
                    \node[job_m,fill=cj1] (41) at (4+4-1,-5) {$2m\qquad\qquad$};
                    
                    \foreach \x in {1,2,3}{
                        \node[job_m,fill=cj1] (41) at (8+\x-1,-\x) {};
                    }
				\end{scope}
				
				\foreach \x in {1,2}{
					\node[minimum width=\x*0.75cm, job_d,pattern=north east lines] (41) at (0,-\x-1) {};
				}

				\node[minimum width=3*0.75cm, job_d,pattern=north east lines] (41) at (0,-5) {};
								
				
				\node[minimum width=0.75cm,job_d,pattern=north east lines] (leg_d) at (0, -7) {};
				\node[right of = leg_d,anchor=west,xshift=-0.5cm] (leg_d_txt)  {Dummy jobs};
				
				\node[rectangle,draw,anchor=west,minimum height=0.75cm,minimum width=0.75cm, fill=cj1] (leg_r) at (6, -7) {$j$};
				\node[right of = leg_r,anchor=west,xshift=-0.5cm] (leg_r_txt)  {Real jobs};
				
				\node (a) at (1,0.5) {};
				
   			\draw[-,thick] (3,-\m+0.1) -- node[below,yshift=-1mm] {$m-1$} (3, -\m-0.1);
			\draw[-,thick] (7,-\m+0.1) -- node[below,yshift=-1mm] {$2m-1$} (7, -\m-0.1);
				
	\end{tikzpicture}}
	
\end{figure}

\begin{corollary}\label{cor:CompletionTime_I_bar_OPT}
	There is a WSPT ordering of the jobs of instance $\bar{I}$ such that the completion times of the real jobs are equal to those of an optimal schedule of instance $I$. 
\end{corollary}
\begin{proof}
	Reverse the order of the real jobs in the proof of Lemma~\ref{lem:CompletionTime_I_bar}.
\end{proof}

In what follows, we aim to offset the weighted completion time of the dummy jobs by splitting some real jobs. 
One can easily verify that for an ordering as described in Lemma~\ref{lem:CompletionTime_I_bar}, the weighted completion time of the dummy jobs is $\Delta \coloneqq \sum_{j=x+B+1}^{x+B+m-1}\bar{w}_jC_j\big(WSPT(\bar{I})\big) = \sum_{i=1}^{m-1} i^2 = \frac{m(m-1)(2m-1)}{6}$. 
To offset this increase in the total weighted completion time, we can split up some real jobs with processing times $m$. 
Specifically, we split a job $j_0$ with $\bar{w}_{j_0}=\bar{p}_{j_0}$ into two jobs $j_1$ and $j_2$ with $\bar{w}_{j_1} = \bar{p}_{j_1}, \ \bar{w}_{j_2} = \bar{p}_{j_2}$ and $\bar{p}_{j_1}+\bar{p}_{j_2} = \bar{p}_{j_0}$, see Figure~\ref{fig:Schematic_Display_jobSplitting}. 
Such a split reduces the total weighted completion time of a schedule by $\bar{p}_{j_1}\bar{p}_{j_2}$ \citep{Schwiegelshohn2011}. 

Similarly to Lemma~\ref{lem:SmallAndBigJobs}, this procedure can be repeated until all ``split'' jobs starting before time~$m-1$ have a processing time of 1. Such an instance is illustrated in Figure~\ref{fig:Schematic_Display_afterSplitting}. 
One can easily verify that these splits result in a reduction in the total weighted completion time of $\sum_{i=1}^{m-1} i^2 = \Delta$.
We denote the instance in which the jobs have been split up as $\bar{I}'$.

\begin{figure}
	\centering
	\caption{Schematic display of job splitting}\label{fig:Schematic_Display_jobSplitting}
	\subfloat[Original job]{
	\scalebox{0.85}{
		\begin{tikzpicture}[scale=0.75]
			\tikzset{job_m/.style={rectangle,draw,anchor=west,minimum height=0.75cm,minimum width=5*.75cm}}			
			\tikzset{job_mS/.style={rectangle,draw,anchor=west,minimum height=0.75cm,minimum width=2.5*.75cm}}			
			\tikzset{job_d/.style={rectangle,draw,anchor=west,minimum height=0.75cm}}			
			\pgfmathsetmacro{\m}{1.65}
			
			\draw[-latex] (0, -\m) -- node[pos=1,right] {$t$} (6.5, -\m);
		
			\node[job_m,fill=cj1] (41) at (0.5,-1) {$j_0$};
			\draw[-,thick] (5.5,-\m+0.1) -- node[below,yshift=-1mm] {$C_{j_0}$} (5.5, -\m-0.1);			
	\end{tikzpicture}}}
 
	\subfloat[Split job]{
		\scalebox{0.85}{
			\begin{tikzpicture}[scale=0.75]
				\tikzset{job_m/.style={rectangle,draw,anchor=west,minimum height=0.75cm,minimum width=5*.75cm}}			
				\tikzset{job_mS/.style={rectangle,draw,anchor=west,minimum height=0.75cm,minimum width=2*.75cm}}		
                \tikzset{job_mSS/.style={rectangle,draw,anchor=west,minimum height=0.75cm,minimum width=3*.75cm}}			
				\tikzset{job_d/.style={rectangle,draw,anchor=west,minimum height=0.75cm}}			
				\pgfmathsetmacro{\m}{1.65}
				
				\draw[-latex] (0, -\m) -- node[pos=1,right] {$t$} (6.5, -\m);
				
				\node[job_mS,fill=cj1] (41) at (0.5,-1) {$j_1$};
				\node[job_mSS,fill=cj1] (41) at (2.5,-1) {$j_2$};
				\draw[-,thick] (2.5,-\m+0.1) -- node[below,yshift=-1mm] {$C_{j_1}$} (2.5, -\m-0.1);
				\draw[-,thick] (5.5,-\m+0.1) -- node[below,yshift=-1mm] {$C_{j_2}$} (5.5, -\m-0.1);
				
		\end{tikzpicture}}
	}
\end{figure}
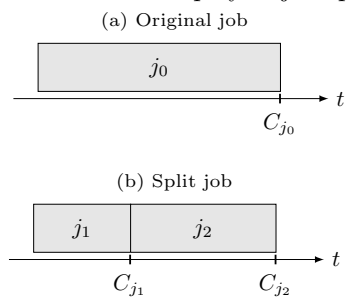

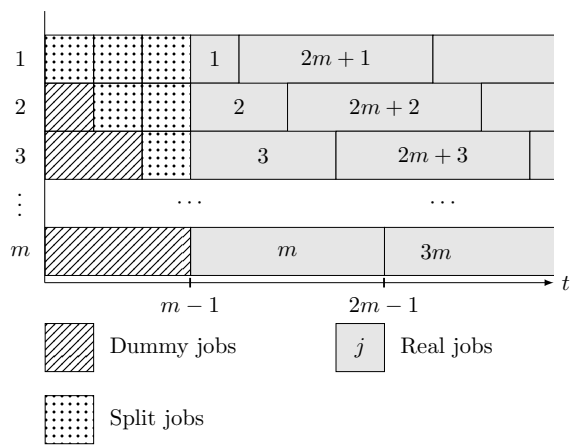
\begin{figure}
	\centering
	\caption{Schematic Gantt chart for the worst-case WSPT schedule after job splitting}\label{fig:Schematic_Display_afterSplitting}
	\scalebox{0.85}{
		\begin{tikzpicture}[scale=0.75]
			\tikzset{job_m/.style={rectangle,draw,anchor=west,minimum height=0.75cm,minimum width=4*.75cm}}			
			\tikzset{job_mS/.style={rectangle,draw,anchor=west,minimum height=0.75cm,minimum width=.75cm, pattern=dots}}			
			\tikzset{job_d/.style={rectangle,draw,anchor=west,minimum height=0.75cm}}			
			\pgfmathsetmacro{\m}{5.65}
			
			\draw[-] (0,0) -- (0, -\m);
			\draw[-latex] (0, -\m) -- node[pos=1,right] {$t$} (10.5, -\m);

 			\foreach \x in {1,2,3}{
 				\node (Mone) at (-0.5,-\x) {$\x$};
 			}
 			\node[rotate=90] (Mone) at (-0.5,-4) {$\cdots$};
 			\node (Mone) at (-0.5,-5) {$m$};
   
			\foreach \x in {1,2}{
				\node[job_mS] (41) at (\x-1,-\x) {};
			}
            \foreach \x in {1}{
				\node[job_mS] (41) at (\x,-\x) {};
			}

            \foreach \x in {1,2,3}{
                \node[job_mS] (41) at (2,-\x) {};
            }
           
			\foreach \x in {1,2,3}{
				\node[job_mS,fill=cj1,minimum width=\x*0.75cm] (41) at (3,-\x) {$\x$};
			}
            \node[job_m,fill=cj1] (41) at (4-1,-5) {$m$};

            \node[rectangle,anchor=west,minimum height=0.75cm,minimum width=0.75] (Mone) at (2.55,-4) {$\cdots$};
			
			\begin{scope}
					\clip (0,0) rectangle (10.5,-6.2); 
					\foreach \x in {1,2,3}{
						\node[job_m,fill=cj1] (41) at (4+\x-1,-\x) {$2m+\x$};
					}
					\node (Mone) at (5+4-.75,-4) {$\cdots$};
					\node[job_m,fill=cj1] (41) at (4+4-1,-5) {$3m\qquad\qquad$};
					
					\foreach \x in {1,2,3}{
						\node[job_m,fill=cj1] (41) at (8+\x-1,-\x) {};
					}
			\end{scope}

			\foreach \x in {1,2}{
					\node[minimum width=\x*0.75cm, job_d,pattern=north east lines] (41) at (0,-\x-1) {};
            }

            \node[minimum width=3*0.75cm, job_d,pattern=north east lines] (41) at (0,-5) {};
			
			
			\node[minimum width=0.75cm,job_d,pattern=north east lines] (leg_d) at (0,-7) {};
			\node[right of = leg_d,anchor=west,xshift=-0.5cm] (leg_d_txt)  {Dummy jobs};

            \node[job_mS] (leg_s) at (0, -8.5) {};
			\node[right of = leg_s,anchor=west,xshift=-0.5cm] (leg_s_txt)  {Split jobs};
			
			\node[rectangle,draw,anchor=west,minimum height=0.75cm,minimum width=0.75cm, fill=cj1] (leg_r) at (6, -7) {$j$};
			\node[right of = leg_r,anchor=west,xshift=-0.5cm] (leg_r_txt)  {Real jobs};
			
			\node (a) at (1,0.5) {};
			
   			\draw[-,thick] (3,-\m+0.1) -- node[below,yshift=-1mm] {$m-1$} (3, -\m-0.1);
			\draw[-,thick] (7,-\m+0.1) -- node[below,yshift=-1mm] {$2m-1$} (7, -\m-0.1);
			
	\end{tikzpicture}}
	
\end{figure}

The third step of the proof is to show the worst-case performance ratio of WLRL using instances $\bar{I}, \bar{I}'$, and the worst-case performance ratio $\bar{\Rto}$ of the WSPT rule. 

\begin{theorem}
	The WLRL rule has a worst-case performance ratio of $\frac{1}{2}\big(1+\sqrt{2}\big)$.
\end{theorem}
\begin{proof}
	Consider a potential worst-case instance $I = (x,y_0,\ldots,y_{B-1})$ and the corresponding instances $\bar{I}$ and $\bar{I}'$. 
	Let $C_j\big(WSPT(\bar{I})\big)$ denote the completion times achieved when using the ordering of Lemma~\ref{lem:CompletionTime_I_bar} and let $C_j\big(OPT(\bar{I})\big)$ denote the completion times achieved when using the ordering of Corollary~\ref{cor:CompletionTime_I_bar_OPT}. 
	Note that the true optimal solution to $\bar{I}$ may have an even lower total weighted completion time. The following inequalities hold: 

	\begin{align*}
		\frac{1}{2}\big(1+\sqrt{2}\big) \geq&
        \bar{\Rto}(\bar{I}') \\
        \geq&  \frac{\sum_{j=1}^{x+B+m-1} \bar{w}_jC_j\big(WSPT(\bar{I})\big) -\Delta}{\sum_{j=1}^{x+B+m-1} \bar{w}_jC_j\big(OPT(\bar{I})\big)-\Delta} \\[0.25em]
		= &\frac{\sum_{j=1}^{x+B} \bar{w}_jC_j\big(WSPT(\bar{I})\big) + \Delta-\Delta}{\sum_{j=1}^{x+B} \bar{w}_jC_j\big(OPT(\bar{I})\big)+ \Delta-\Delta} \\[0.25em]
		\geq& \frac{\sum_{j=1}^{x+B} mw_jC_j\big(WLRL(I)\big)}{\sum_{j=1}^{x+B} mw_jC_j\big(OPT(I)\big)} = \Rto(I).
	\end{align*}
\end{proof}

Finally, we note that our worst-case performance ratio is tight. 
This means that there is no value $\alpha<\frac{1}{2}\big(1+\sqrt{2}\big)$ such that for every instance $I$, $\Rto(I) \leq \alpha$. 
This can be seen by analyzing instances with a structure of $x$ small jobs and $B<m$ big jobs that all have the same number of loops $y\geq 2$.
This structure is similar to the structure \cite{Schwiegelshohn2011} identified for the parallel machine problem and the analysis is done analogously.

\section{Fixed-size flow shops}\label{sec:m_fixed}

In this section, we consider the reentrant flow shop problem in which the number of machines $m$ is not part of the problem input but is considered a fixed parameter. 
We show that the problem remains weakly NP-hard under this assumption and present a pseudo-polynomial time algorithm.
Additionally, we show that the problem can be approximated within arbitrary precision in polynomial time. 


\subsection{Weak NP-hardness}

To show that the problem remains NP-hard if the number of machines is fixed, we present a reduction from the classical PARTITION problem \citep{Karp1972}. 

\begin{definition}\label{thm:Def_Partition}
	The decision problem PARTITION is defined as follows. Given positive integers~${a_1,\ldots,a_q}$ and $b= \sum_{j=1}^q a_j$, is there a subset $\Sx\subset \St = \{1,\ldots,q\}$ such that ${\sum_{j\in \Sx} a_j = \frac{b}{2}}$?
\end{definition}

Our reduction uses the fact that we can assume an optimal schedule to be non-interruptive. 
This implies that in an optimal schedule, we have two different progressions, starting at time 0 and at time 1, that form a partition of the jobs. 
The analysis of this reduction is similar to the one from \cite{Lenstra1977} for a related problem.

\begin{theorem}\label{thm:NP_hard_Weighted}
	Minimizing the total weighted completion time for a reentrant flow shop with $m=2$ and unit processing times ($F2|reentry, \ p_{ijk} = 1|\sum w_jC_j$) is NP-hard. 
\end{theorem}
\begin{proof}
	Let $(b,a_1,\ldots,a_q)$ be an instance of PARTITION. Without loss of generality, we can assume~$b$ to be even. 
	We define an instance of the reentrant flow shop problem with unit processing times as $n \coloneqq  q,\  \Lps_j\coloneqq a_j$, and $w_j\coloneqq a_j$ for all $j=1,\ldots,n$. Note that for all jobs $j$, $\frac{w_j}{\Lps_j}=1$. 
	Given an optimal schedule for the created instance, we can assume this schedule to be non-interruptive by Theorem~\ref{thm:NonInterruption}. 
	We define $\Sx$ as the set of jobs that start their loops at even times, i.e., $\Sx= \big \{j: S_{j1}\in \{0,2,4\ldots\}\big\}$. 
	The remaining jobs of set $\St \setminus \Sx$ then start all their loops at odd times, see Figure~\ref{fig:ProofNP_hard_A}. 
	As $\frac{w_j}{\Lps_j}=1$, the total weighted completion time of this schedule does not depend on the order of jobs within sets $\Sx$ and $\St \setminus \Sx$, but only on the assignment of jobs to either set. 
	We refer to this schedule as $\Sched(\Sx)$ with completion times $C_j\big(\Sched(\Sx)\big)$ and to its total weighted completion time as $\OFV(\Sx)=\sum_{j=1}^n w_jC_j\big(\Sched(\Sx)\big)$. 	
	
	Consider another schedule that starts all jobs currently in set $\St \setminus \Sx$ at even times after all jobs of set $\Sx$ are completed; see Figure~\ref{fig:ProofNP_hard_B}. 
	We refer to this schedule as $\Sched(\St)$ with completion times $C_j\big(\Sched(\St)\big)$ and to its total weighted completion time as ${\OFV(\St)=\sum_{j=1}^n w_jC_j\big(\Sched(\St)\big)}$. 	
	Note that $\OFV(\St)$ does not depend on the order of jobs, but only on the given instance of PARTITION. Therefore, it can be seen as constant in the following. 
    We can rewrite $\OFV(\St)$ as follows:

	\begin{align*}
		\OFV(\St) =& \sum_{j \in \Sx} w_j C_j\big(\Sched(\St)\big) + \sum_{j \in \St \setminus \Sx} w_j C_j\big(\Sched(\St)\big)&\\[0.25em]
		=& \sum_{j \in \Sx} w_j C_j\big(\Sched(\Sx)\big) \\
        & + \sum_{j \in \St \setminus \Sx} w_j \big[ C_j\big(\Sched(\Sx)\big) + 2\sum_{i \in \Sx} \Lps_i -1\big]&\\[0.25em]
		=&\OFV(\Sx) + \sum_{j \in \St \setminus \Sx} w_j \left(2\sum_{i \in \Sx} \Lps_i\right) -\sum_{j \in \St \setminus \Sx} w_j&\\[0.25em]
		=&\OFV(\Sx) + 2\left(\sum_{j \in \Sx} a_j\right)\left(\sum_{j \in \St \setminus \Sx} a_j\right) \\
        & -\sum_{j \in \St \setminus \Sx} a_j.&\\[0.25em]
    \end{align*}

    We define ${c\coloneqq \sum_{j \in \Sx} a_j - \frac{1}{2}b} \in \mathbb{Z}$ and the constant ${y \coloneqq \OFV(\St) - \frac{1}{2}b^2+\frac{1}{2}b}$. Together with the above we get:

    \begin{align*}
		\OFV(\Sx) =& \OFV(\St) - 2\left(\frac{1}{2}b+c\right)\left(\frac{1}{2}b-c\right) +\frac{1}{2}b-c&\\[0.25em]
		=& \OFV(\St) - \frac{1}{2}b^2+\frac{1}{2}b + 2c^2 -c&\\[0.25em]
        =& y + 2c^2 - c.&
	\end{align*}
	
	It follows that the instance of PARTITION is a yes-instance if and only if $\OFV(\Sx)\leq y$. 
\end{proof}

\begin{figure*}
	\centering
	\caption{Schematic displays of the schedules in the proof of Theorem~\ref{thm:NP_hard_Weighted}.\label{fig:ProofNP_hard}}
	\subfloat[Optimal schedule $\Sched(\Sx)$ \label{fig:ProofNP_hard_A}]{
		\scalebox{0.85}{
			\begin{tikzpicture}[scale=0.75]
				
				\pgfmathsetmacro{\m}{2.65}
				
				
				\node (Machines) at (-1,0) {\small Machine};
				\foreach \x in {1,2}{
					\node (Mone) at (-1,-\x) {$\x$};
				}
				
				\draw[-] (0,0) -- (0, -\m);
				\draw[-latex] (0, -\m) -- node[pos=1,below] {} (13, -\m);
				
				\node[job_tr] (41) at (10,-1) {\small $\ldots$};
				\node[job_tr] (41) at (11,-2) {\small $\ldots$};
				
				\foreach \x in {0,2,4,6,8}{
					\node[job1] (41) at (\x,-1) {};
					\node[job2] (41) at (\x+1,-1) {};
					\node[job1] (41) at (\x+1,-2) {};
					\node[job2] (41) at (\x+2,-2) {};
				}
				
				
				\node[job1] (leg_d) at (14,-0.5) {};
				\node[right of = leg_d,anchor=west,xshift=-0.5cm] (leg_d_txt)  {$\Sx$};
				
				\node[job2] (leg_r) at (14,-1.7) {};
				\node[right of = leg_r,anchor=west,xshift=-0.5cm] (leg_r_txt)  {$\St \setminus \Sx$};
				
				\node (a) at (1,0.5) {};
				
	\end{tikzpicture}}}%
	
	\subfloat[Schedule $\Sched(\St)$ \label{fig:ProofNP_hard_B}]{
		\scalebox{0.85}{
			\begin{tikzpicture}[scale=0.75]

				\pgfmathsetmacro{\m}{2.65}
				
				\draw[-] (0,0) -- (0, -\m);
				\draw[-latex] (0, -\m) -- node[pos=1,below] {} (13, -\m);
				
				\node (Machines) at (-1,0) {\small Machine};
				\foreach \x in {1,2}{
					\node (Mone) at (-1,-\x) {$\x$};
				}
				
				
				\node[job_tr] (41) at (4,-1) {\small $\ldots$};
				\node[job_tr] (41) at (4,-2) {\small $\ldots$};
				
				\node[job_tr] (41) at (10,-1) {\small $\ldots$};
				\node[job_tr] (41) at (11,-2) {\small $\ldots$};
				
				\foreach \x in {0,2,5}{
					\node[job1] (41) at (\x,-1) {};
					\node[job1] (41) at (\x+1,-2) {};
				}
				\foreach \x in {7,9}{
					\node[job2] (41) at (\x,-1) {};
					\node[job2] (41) at (\x+1,-2) {};
				}			
				
				\node[job1] (leg_d) at (14,-0.5) {};
				\node[right of = leg_d,anchor=west,xshift=-0.5cm] (leg_d_txt)  {$\Sx$};
				
				\node[job2] (leg_r) at (14,-1.7) {};
				\node[right of = leg_r,anchor=west,xshift=-0.5cm] (leg_r_txt)  {$\St \setminus \Sx$};
				
				\node (a) at (1,0.5) {};
		\end{tikzpicture}}	
	}
\end{figure*}


In what follows, we present a pseudo-polynomial time algorithm for ${Fm|reentry, \ p_{ijk} = 1|\sum w_jC_j}$ that uses dynamic programming techniques. 
An algorithm is said to have pseudo-polynomial runtime if its runtime is polynomial in the size of the input in unary encoding (see, e.g., \citealp{Lawler1977,Garey1978strong,Rothkopf1966, Lawler1969, Sahni1976}). 
Due to Theorem~\ref{thm:NonInterruption}, we can limit our analysis to non-interruptive schedules. 
This implies that once the first loop of a job is scheduled, we may assume that the remaining loops of this job are scheduled within the same progression without any interruption. 
Furthermore, note that jobs assigned to the same progression must be scheduled in decreasing order of $\frac{w_j}{\Lps_j}$ (this can be shown via an interchange argument). 
These two observations significantly reduce the class of schedules to be considered. 
In what follows, we assume the jobs to be indexed so that $\frac{w_1}{\Lps_1} \geq \frac{w_2}{\Lps_2} \geq \ldots \geq\frac{w_n}{\Lps_n} $ holds.

Consider now the $m$ progressions defined by the jobs with starting times~$0,1,\ldots,m-1$. We denote these progressions as $\Pg_1,\ldots,\Pg_m$. 
It can be easily verified that $m\sum_{j=1}^n \Lps_j$ is an upper bound on the completion time of the last job in any optimal schedule. 
We define a function $\FDP_j: \{0,1,\ldots,m\sum_{\bar{j}=1}^n \Lps_{\bar{j}}\}^m \rightarrow \mathbb{N}_0 \cup \{\infty\}$ for each $j = 1,\ldots,n$. 
The value $\FDP_j(t_1,\ldots,t_m)$ shall represent the minimum total weighted completion time when jobs $1,\ldots,j$ are scheduled such that the last job in progression $\Pg_i$ is completed at time $t_i$ for $i = 1,\ldots,m$. 
Further, we define the function ${\FDP_0: \{0,1,\ldots,m\sum_{j=1}^n \Lps_j\}^m \rightarrow \{0,\infty\}}$ with the initial condition: $\FDP_0(0, 1, \ldots, m-1)=0$, and $\FDP_0(t_1,\ldots,t_m)=\infty$ for all other vectors $(t_1,\ldots,t_m)\in \{0,1,\ldots,m\sum_{j=1}^n \Lps_j\}^m$.

The following dynamic programming relationship holds for all $j=1,\ldots,n$:
\begin{multline*} 
\FDP_j(t_1,\ldots,t_m) \\
= \min_{i=1,\ldots,m} \big(w_jt_i + \FDP_{j-1}(t_1,\ldots,t_i-m\Lps_j,\ldots,t_m)\big).
\end{multline*}

These dynamic programming equations can be solved in runtime $O\big(mn(m\sum_{j=1}^n \Lps_j)^m\big)$. 
As we consider~$m$ to be a constant in this section, this algorithm is polynomial in the number of jobs $n$ and in the total number of loops $\sum_{j=1}^n \Lps_{j}$. 
In practice, tighter bounds on the runtime may be established, e.g., by finding a tighter upper bound on the last completion time of a job in an optimal schedule. 

We note that the idea of our dynamic programming algorithm is similar to the dynamic programming approaches for parallel machine scheduling problems outlined in \cite{Rothkopf1966} and \cite{Lawler1969}. 
The major difference of our algorithm lies in the fact that, instead of considering $m$ parallel machines that are available at time $0$, we consider progressions $\Pg_1,\ldots,\Pg_m$ starting at times $0,1,\ldots,m-1$ and loops of jobs that are processed without interruptions. 


\subsection{A fully polynomial-time approximation scheme}

In this subsection, we present a fully polynomial-time approximation scheme (FPTAS) for reentrant flow shops with unit processing times and the total weighted completion time objective. 
An FPTAS is defined as a family of algorithms with performance ratio $1 + \varepsilon$ for all $\varepsilon > 0$ such that their runtime is polynomial in the size of the input and in $\frac{1}{\varepsilon}$. 
Similar to the worst-case analysis of the WLRL rule, we use dummy jobs to transform an instance of $Fm|reentry, \ p_{ijk} = 1|\sum w_jC_j$ into an instance of $Pm||\sum w_j C_j$. Then, we apply the FPTAS for $Pm||\sum w_j C_j$ proposed by \cite{Sahni1976}. 

The following definition shows the transformation. 

\begin{definition}\label{def:Trafo_FPTAS}
	For an instance $I$ of $Fm| reentry, p_{ijk}=1| \sum w_j C_j$, we create an instance $\bar{I}$ of $P_m|| \sum_j w_j C_j$ as follows: $\bar{p}_j \coloneqq m\Lps_{j}, \ \bar{w}_j \coloneqq w_{j},  \ j=1,\ldots,n$. 
	Let $\bar{w}_{\max}\coloneqq \max_{j=1,\ldots,n} \bar{w}_j$ denote the maximum weight of the jobs. 
	We additionally define $m-1$ dummy jobs with $\bar{p}_{n+i}\coloneqq i$ and $\bar{w}_j \coloneqq  \bar{w}_{\max}, \ j = n+1,\ldots,n+m-1$. 
\end{definition}

The following lemma shows that we can assume that no two dummy jobs are processed on the same machine. 

\begin{lemma}\label{lem:FPTAS_swap}
	For a feasible schedule $\Sched$ of $\bar{I}$, we can find in $O(n \log n)$ a schedule $\Sched'$ such that no two dummy jobs are assigned to the same machine, the jobs are sorted according to the WSPT rule on each machine, and $\sum \bar{w}_jC_j(\Sched') \leq \sum \bar{w}_jC_j(\Sched)$. 
\end{lemma}

\begin{proof}
	For jobs that are scheduled on the same machine, it is well known that sorting them according to the WSPT rule does not increase the total weighted completion time \citep{Smith1956}. 
	As $\bar{w}_{j}= w_{\max}$ and $\bar{p}_{n+i}<\bar{p}_j$ for all $j=1,\ldots,n$ and $i=1,\ldots m-1$, we can assume that on each machine, all dummy jobs are completed before the first ``real'' job starts its processing. 
	
	Assume that in schedule $\Sched$, there are two dummy jobs, $n+j_1$ and~$n+j_2$, processed as the first two jobs on the same machine $i$. 
	As there are $m-1$ dummy jobs, there must be a machine $i'$ on which no dummy job is processed. 
	Denote as $\Jbs_i$ the set of jobs scheduled on machine $i$ apart from $n+j_1$ and~$n+j_2$.
	Denote as $j'_1$ the first job scheduled on machine $i'$ and the remaining jobs on machine $i'$ as $\Jbs_{i'}$, respectively. 
	
	Consider the following two cases:

    \pmb{(i) $\sum_{j \in \Jbs_i} \bar{w}_j \geq \sum_{j \in \Jbs_{i'}} \bar{w}_j$}. We schedule dummy job $n + j_1$ as the first job on machine $i'$. 
	The change in the total completion time can be computed as follows:
    \begin{align*}
		& \bar{p}_{n+j_1} \left[ \left(\bar{w}_{j'_1}+\sum_{j \in \Jbs_{i'}} \bar{w}_j\right) -   \left(\bar{w}_{n+j_2}+\sum_{j \in \Jbs_{i}} \bar{w}_j\right)\right]\\
		& = \bar p_{i_1} \left(\bar{w}_{j'_1}-w_{\max} + \sum_{j \in \Jbs_{i'}} \bar{w}_j  -   \sum_{j \in \Jbs_{i}} \bar{w}_j\right) \leq 0.
	\end{align*}
	
    \pmb{(ii) $\sum_{j \in \Jbs_i} \bar{w}_j < \sum_{j \in \Jbs_{i'}} \bar{w}_j$}. We interchange dummy job~$n + j_2$ with job~$j'_1$. 
	The change in the total completion time can be computed as follows:
	\begin{align*}
		& \bar{p}_{n+j_1}\left(\bar{w}_{j'_1} - \bar{w}_{n+j_2}\right) \\
        & +  \left(\bar{p}_{j'_1} - \bar{p}_{n+j_2}\right)\left(\sum_{j \in \Jbs_{i}} \bar{w}_j - \sum_{j \in \Jbs_{i'}} \bar{w}_j\right) < 0.
	\end{align*}
	
	The described interchanges can be performed in constant time for fixed $m$ while sorting each machine by non-increasing ratios~$\frac{w_j}{p_j}$ may require time $O(n \log n)$.
		
\end{proof}

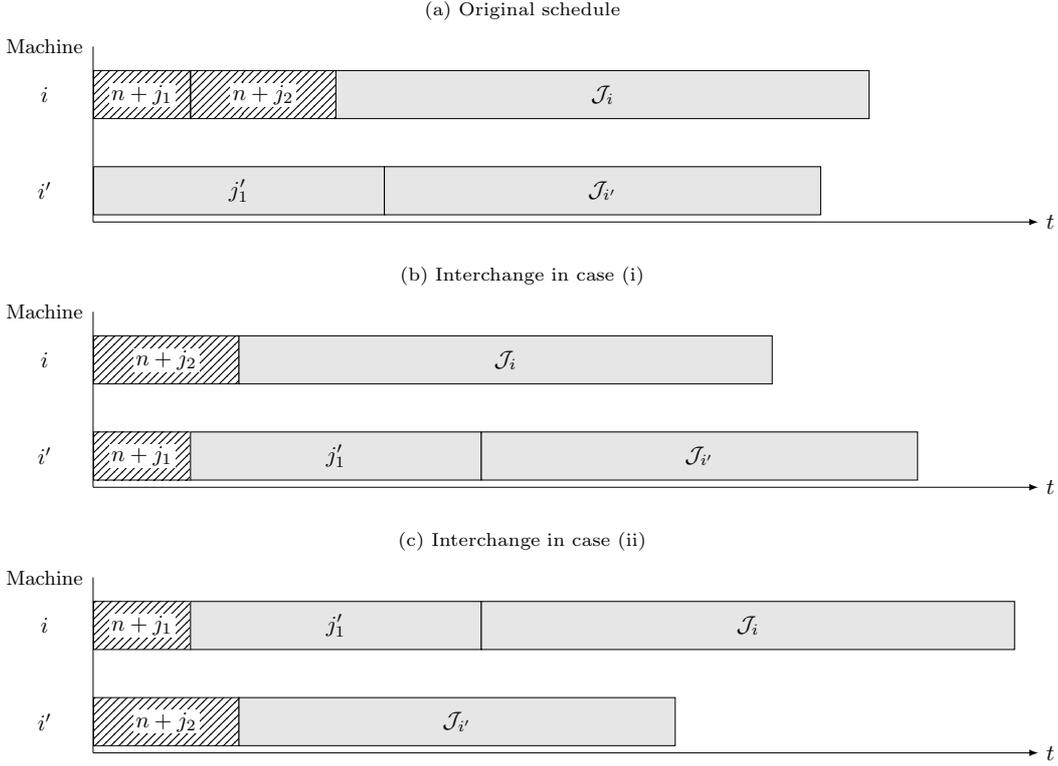
\begin{figure*}
	\centering
	\caption{Schematic Gantt chart for the interchanges in the proof of Lemma~\ref{lem:FPTAS_swap}}\label{fig:FPTAS_swap}
	\subfloat[Original schedule]{
	\scalebox{0.85}{
		\begin{tikzpicture}[scale=0.75]
			\tikzset{job_m/.style={rectangle,draw,anchor=west,minimum height=0.75cm,minimum width=5*.75cm}}			
			\tikzset{job_d/.style={rectangle,draw,anchor=west,minimum height=0.75cm}}			
			\pgfmathsetmacro{\m}{3.65}
			
			\draw[-] (0,0) -- (0, -\m);
			\draw[-latex] (0, -\m) -- node[pos=1,right] {$t$} (19.5, -\m);
			\node (Machines) at (-1,0) {\small Machine};
			\node (Mone) at (-1,-1) {$i$};
			\node (Mone) at (-1,-3) {$i'$};
						
			\node[minimum width=2*0.75cm, job_d,pattern=north east lines] (41) at (0,-1) {};
			\node[lab] at (41) {$n+j_1$};			
			\node[minimum width=3*0.75cm, job_d,pattern=north east lines] (41) at (2,-1) {};
			\node[lab] at (41) {$n+j_2$};
			
			\node[job_m, minimum width=11*0.75cm, fill=cj1] (41) at (5,-1) {};
			\node[] at (41) {$\Jbs_i$};
			
			\node[job_m, minimum width=6*0.75cm, fill=cj1] (41) at (0,-3) {};
			\node[] at (41) {$j'_{1}$};
			\node[job_m, minimum width=9*0.75cm, fill=cj1] (41) at (6,-3) {};
			\node[] at (41) {$\Jbs_{i'}$};
				
	\end{tikzpicture}}}
	
	\subfloat[Interchange in case (i)]{
		\scalebox{0.85}{
			\begin{tikzpicture}[scale=0.75]
				\tikzset{job_m/.style={rectangle,draw,anchor=west,minimum height=0.75cm,minimum width=5*.75cm}}			
				\tikzset{job_d/.style={rectangle,draw,anchor=west,minimum height=0.75cm}}			
				\pgfmathsetmacro{\m}{3.65}
				
				\draw[-] (0,0) -- (0, -\m);
				\draw[-latex] (0, -\m) -- node[pos=1,right] {$t$} (19.5, -\m);
				\node (Machines) at (-1,0) {\small Machine};
				\node (Mone) at (-1,-1) {$i$};
				\node (Mone) at (-1,-3) {$i'$};
				
				\node[minimum width=3*0.75cm, job_d,pattern=north east lines] (41) at (0,-1) {};
				\node[lab] at (41) {$n+j_2$};
				
				\node[job_m, minimum width=11*0.75cm, fill=cj1] (41) at (3,-1) {};
				\node[] at (41) {$\Jbs_i$};
				
				\node[minimum width=2*0.75cm, job_d,pattern=north east lines] (41) at (0,-3) {};
				\node[lab] at (41) {$n+j_1$};			
				
				\node[job_m, minimum width=6*0.75cm, fill=cj1] (41) at (2,-3) {};
				\node[] at (41) {$j'_{1}$};
				\node[job_m, minimum width=9*0.75cm, fill=cj1] (41) at (8,-3) {};
				\node[] at (41) {$\Jbs_{i'}$};
				
	\end{tikzpicture}}}

	\subfloat[Interchange in case (ii)]{
		\scalebox{0.85}{
			\begin{tikzpicture}[scale=0.75]
				\tikzset{job_m/.style={rectangle,draw,anchor=west,minimum height=0.75cm,minimum width=5*.75cm}}			
				\tikzset{job_d/.style={rectangle,draw,anchor=west,minimum height=0.75cm}}			
				\pgfmathsetmacro{\m}{3.65}
				
				\draw[-] (0,0) -- (0, -\m);
				\draw[-latex] (0, -\m) -- node[pos=1,right] {$t$} (19.5, -\m);
				\node (Machines) at (-1,0) {\small Machine};
				\node (Mone) at (-1,-1) {$i$};
				\node (Mone) at (-1,-3) {$i'$};
				
				\node[minimum width=2*0.75cm, job_d,pattern=north east lines] (41) at (0,-1) {};
				\node[lab] at (41) {$n+j_1$};			
				\node[job_m, minimum width=6*0.75cm, fill=cj1] (41) at (2,-1) {};
				\node[] at (41) {$j'_{1}$};
				
				\node[job_m, minimum width=11*0.75cm, fill=cj1] (41) at (8,-1) {};
				\node[] at (41) {$\Jbs_i$};

				\node[minimum width=3*0.75cm, job_d,pattern=north east lines] (41) at (0,-3) {};
				\node[lab] at (41) {$n+j_2$};
				\node[job_m, minimum width=9*0.75cm, fill=cj1] (41) at (3,-3) {};
				\node[] at (41) {$\Jbs_{i'}$};
				
	\end{tikzpicture}}}
\end{figure*}

Due to the previous lemma, we know that we can apply a simple interchange on schedules generated by Sahni's FPTAS to compute schedules in which all dummy jobs start processing at time $0$ on separate machines. 
This allows us to devise an FPTAS for the reentrant flow shop problem. 

\begin{theorem}\label{thm:FPTAS}
	There exists an FPTAS for $Fm| reentry, p_{ijk}=1| \sum w_j C_j$.
\end{theorem}
\begin{proof}
	Let $\varepsilon>0$ be arbitrary but fixed. For an instance $I$ of $Fm| reentry, p_{ijk}=1| \sum w_j C_j$, create an instance $\bar{I}$ of $Pm||\sum w_jC_j$ as described in Definition~\ref{def:Trafo_FPTAS}. 
	Let $\OFVV\big(OPT(I)\big)$ and $\OFVV\big(OPT(\bar{I})\big)$ denote the optimal objective function values of instances $I$ and $\bar{I}$, respectively. 
	We apply Sahni's FPTAS and the interchanges described in Lemma~\ref{lem:FPTAS_swap} to instance $\bar{I}$ to obtain a schedule $\Sched(\bar{I})$ in time $O\big(n(\frac{n^2}{\varepsilon})^{m-1} + n \log n\big)$. 
	For this schedule, $\sum_{j=1}^{n+m-1} \bar{w}_jC_j\big(\Sched(\bar{I})\big) \leq (1 + \varepsilon) \OFVV\big(OPT(\bar{I})\big)$ holds. 
	We define a schedule $\Sched(I)$ for instance $I$ in which all jobs have the same completion times as the ``real'' jobs in schedule $\Sched(\bar{I})$. 	
	Additionally, define $\Delta \coloneqq w_{\max}\sum_{i=1}^m (i-1) = w_{\max}\frac{m\cdot(m-1)}{2}$.
	We can express the total weighted completion time of $\Sched$ as $\sum_{j=1}^n w_jC_j\big(\Sched(I)\big) = \sum_{j=1}^n \bar{w}_jC_j\big(\Sched(\bar{I})\big) + \Delta$. 
	As the optimal value $\OFVV\big(OPT(I)\big)$ is lower bounded by $\sum_{j = 1}^n w_j m\Lps_j \geq mw_{\max}$, we have:
	\filbreak
    \begin{align*}
		\sum_{j=1}^n w_jC_j\big(\Sched(I)\big) & \leq (1 + \varepsilon) \OFVV\big(OPT(I)\big) + \varepsilon \Delta \\
        & \leq (1 + \varepsilon \frac{m+1}{2}) \OFVV\big(OPT(I)\big).
	\end{align*}
	As $m$ is constant, this completes the proof.
\end{proof}

\section{Conclusion and outlook}\label{sec:Conclusion}

In this paper, we considered scheduling a set of jobs that must go through multiple loops in a flow shop with the objective of minimizing the total (weighted) completion time.  
We introduced non-interruptive schedules and showed that there is a non-interruptive schedule that minimizes the total weighted completion time. 
We showed that the general problem of minimizing the total weighted completion time is strongly NP-hard.
We introduced the priority rule LRL and showed that it minimizes the total completion time and the total weighted completion time if the weights are agreeable. 
We introduced the priority rule WLRL and provided a tight performance guarantee of about 1.2. 
Additionally, we analyzed reentrant flow shops with a fixed number of machines.
We presented a dynamic programming algorithm that has pseudo-polynomial runtime in the total number of loops and an FPTAS.

For future research, we propose to examine the total completion time objective for reentrant flow shops with machine-ordered and proportionate processing times. 
For both these processing time structures, one can easily create an instance in which optimal schedules allow for unforced idleness. 
For machine-ordered processing times, one could analyze a modification of the WLRL priority rule that alternates between scheduling jobs according to WLRL and waiting for a job to complete its current loop. 
For proportionate processing times, one can show that the problems of minimizing the makespan and minimizing the total (unweighted) completion time are strongly NP-hard by a reduction from 3-PARTITION.
One could analyze a generalization of the WLRL rule that takes into account the processing times of jobs, i.e., priority values~$\frac{w_j}{p_j\cdot \Lps_j}$. 
In a computational experiment, this rule's empirical average performance ratio was about 1.15, while the empirical worst case was about 2.09. 
It would be interesting to investigate whether a constant bound on the worst-case performance ratio of this priority rule can be established. 
Additionally, it should be analyzed how the WLRL priority rule can be embedded in more complex scheduling heuristics for reentrant shops with arbitrary processing times. 
For these problems, combining priority rules with constraint programming may be a promising direction. 

Another direction for future research is stochastic flow shops with reentry. 
In addition to stochastic processing times, which have been studied for conventional flow shops (see, e.g., \citealp{Pinedo2022,Emmons2012}), one could study stochastic reentries. 
In real-world production systems, uncertainty in reentries may arise if a job must undergo repair or rework steps after a defect has been detected. 
For such a setting, conditions could be established under which the ''Most Expected Remaining Loops First'' rule minimizes the expected makespan or the "Weighted Least Expected Remaining Loops First'' rule minimizes the expected total weighted completion time.
Also, one could investigate performance bounds for these priority rules that may depend on the coefficients of variation of the underlying distributions. 



\bmhead{Acknowledgements}
Nicklas Klein expresses his gratitude to the Hasler Foundation for supporting a research visit during which this paper was initiated. 

The authors Maximilian von Aspern and Felix Buld were funded by the Deutsche Forschungsgemeinschaft (DFG, German Research Foundation) - project number 277991500.
%


\setlength{\bibsep}{0pt plus 0.3ex}

\bibliography{bib_flowshops_reentry}



\end{document}